\documentclass[12pt]{amsart}

\usepackage{enumerate}
 
\usepackage[active]{srcltx}
\usepackage{nicefrac}

\usepackage{amsmath}

\allowdisplaybreaks[4]

\usepackage{amsthm,amssymb,setspace}
\usepackage[pagebackref,colorlinks,linkcolor=red,citecolor=blue,urlcolor=blue,hypertexnames=true]{hyperref}
\usepackage{amsrefs}
\usepackage[matrix, arrow]{xy}

\newcommand{\C}{\mathbb C}

\newcommand{\Q}{\mathbb Q}

\newcommand{\Z}{\mathbb Z}

\theoremstyle{plain}
\newtheorem{theorem}{Theorem}[section]
\newtheorem{corollary}[theorem]{Corollary}
\newtheorem{lemma}[theorem]{Lemma}

\theoremstyle{definition}

\theoremstyle{remark}

\newtheorem{remark}[theorem]{Remark}

\begin{document}

\onehalfspace

\title{Maximal Haagerup subalgebras in $L(\Z^2\rtimes SL_2(\Z))$}

\author{Yongle Jiang}
\address{Yongle Jiang, School of Mathematical Sciences, Dalian University of Technology, Dalian, 116024, China}
\email{yonglejiang@dlut.edu.cn}

\begin{abstract}
We prove that $L(SL_2(\textbf{k}))$ is a maximal Haagerup von Neumann subalgebra in $L(\textbf{k}^2\rtimes SL_2(\textbf{k}))$ for $\textbf{k}=\Q$. Then we show how to modify the proof to handle $\textbf{k}=\Z$. The key step for the proof is a complete description of all intermediate von Neumann subalgebras between $L(SL_2(\textbf{k}))$ and $L^{\infty}(Y)\rtimes SL_2(\textbf{k})$, where $SL_2(\textbf{k})\curvearrowright Y$ denotes the quotient of the algebraic action $SL_2(\textbf{k})\curvearrowright \widehat{\textbf{k}^2}$ by modding out the relation $\phi\sim \phi'$, where $\phi$, $\phi'\in \widehat{\textbf{k}^2}$ and $\phi'(x, y):=\phi(-x, -y)$ for all $(x, y)\in \textbf{k}^2$. As a by-product, we show $L(PSL_2(\Q))$ is a maximal von Neumann subalgebra in $L^{\infty}(Y)\rtimes PSL_2(\Q)$; in particular, $PSL_2(\Q)\curvearrowright Y$ is a prime action, i.e. it admits no non-trivial quotient actions.
\end{abstract}



\maketitle

\tableofcontents

\section{Introduction}\label{section: introduction}

Let $N\subset M$ be an inclusion of finite von Neumann algebras. Recall that $N$ is a \textit{maximal Haagerup von Neumann subalgebra} in $M$ if $N$ has Haagerup property \cite{choda_m, jolissaint} and every von Neumann subalgebra of $M$ which contains $N$ as a proper subalgebra does not have Haagerup property. 

In \cite{js}, we initiated the study of maximal Haagerup von Neumann subalgebras. One initial motivation for this study is the hope that this may provide a new angle to study non-Haagerup von Neumann algebras, e.g. (diffuse) von Neumann algebras with property (T) \cite{connes_jones}. As is well known, free group factors and von Neumann algebras with property (T) are arguably two most important classes of von Neumann algebras after the intensive studies of amenable ones \cite{connes_annals, haagerup_31}. For free group factors, one of the first non-trivial structure results on their von Neumann subalgebras is due to Popa. In 1980s, he proved that generator masas in free group factors are maximal amenable \cite{popa_injective}, solving a long-standing open question asked by Kadison. By analogy, one may ask what one can say about maximal Haagerup von Neumann subalgebras inside a given von Neumann algebra with property (T) or more generally a von Neumann algebra without Haagerup property, e.g. $L(\mathbb{Z}^2\rtimes SL_2(\mathbb{Z}))$.

In \cite{js}, we presented several concrete examples of maximal Haagerup von Neumann subalgebras. For example, if $H$ denotes an infinite maximal amenable subgroup of $SL_2(\mathbb{Z})$ containing the matrix $\big(\begin{smallmatrix}
1&1\\
1&2
\end{smallmatrix}\big)$, then $L(\mathbb{Z}^2\rtimes H)$ is maximal Haagerup in $L(\mathbb{Z}^2\rtimes SL_2(\mathbb{Z}))$ \cite[Theorem 3.1]{js}. One key ingredient for this is the dichotomy result on ergodic subequivalence relations for the natural action $SL_2(\mathbb{Z})\curvearrowright \widehat{\mathbb{Z}^2}\cong \mathbb{T}^2$ due to Ioana \cite{ioana_adv}. Distinguished from the above (amenable) subgroup $\mathbb{Z}^2\rtimes H$, $SL_2(\mathbb{Z})$ is another maximal Haagerup subgroup inside $\mathbb{Z}^2\rtimes SL_2(\mathbb{Z})$. In fact, we have classified all maximal Haagerup subgroups of $\Z^2\rtimes SL_2(\Z)$ into two distinct classes in \cite[Theorem 2.12]{js}. According to this classification, each one of the above two subgroups is a typical representative for one class respectively. Therefore, it is natural to ask whether $L(SL_2(\mathbb{Z}))$ is also maximal Haagerup in $L(\mathbb{Z}^2\rtimes SL_2(\mathbb{Z}))$. 

Although the abovementioned question was left open in \cite[Problem 5.3]{js}, we have shown several modified versions of the inclusion $SL_2(\mathbb{Z})\subset\mathbb{Z}^2\rtimes SL_2(\mathbb{Z})$ give rise to maximal Haagerup group von Neumann subalgebras \cite[Corollary 3.6, Corollary 3.9]{js}. In this paper, we can answer this question affirmatively.
\begin{theorem}[Corollary \ref{cor: L(SL_2(Z)) is mHAP}]\label{main thm}
$L(SL_2(\mathbb{Z}))$ is a maximal Haagerup von Neumann subalgebra in $L(\mathbb{Z}^2\rtimes SL_2(\mathbb{Z}))$.
\end{theorem}
This theorem may be thought of as the counterpart of Popa's result \cite{popa_injective} after shifting our attention away from maximal amenability and concentrate on maximal Haagerup property. Nevertheless, our proof is different from Popa's which relies on a rather fine analysis on certain relative commutants in the ultrapower of the ambient algebra. Instead, our proof explores a rigid feature on certain von Neumann subalgebras containing $L(SL_2(\Z))$.

Next, let us briefly describe the strategy for the proof. For ease of notation, we take the inclusion $LH\subset LG$ for example to explain the method. Same method also works for the inclusion $LG\subset L^{\infty}(X)\rtimes G$ for certain p.m.p. action $G\curvearrowright X$. 

Let $N:=LH\subset P\subset LG:=M$ be inclusions of von Neumann algebras. Denote by $\tau$ the canonical trace on $LG$, and $E: (LG, \tau)\twoheadrightarrow (P, \tau|_P)$ the trace preserving conditional expectation. 
To prove the above theorem, the natural idea is to completely determine $P$, which is essentially equivalent to determining $\{E(u_g): g\in G\}$. To do this, we think of $\{E(u_g): g\in G\}$ as a set of unknowns and try to find sufficiently many equations which involve these unknowns. In this paper, the following equations are used:

(1) $\phi(E(u_g))=E(\phi(u_g))$ for all $g\in G$, where $\phi\in Aut(LG, P)$, i.e. $\phi$ is an automorphism of $LG$ which fixes $P$ as a set globally, e.g. $\phi=Ad(u)$ for any unitary $u$ in $LH$. 

(2) $E(E(u_s)u_t)=E(u_s)E(u_t)$ for all $s$, $t\in G$.   

Historically speaking, the use of (1) to solve for $\{E(u_g): g\in G\}$ has already appeared in several works \cite{choda_h, chifan_das, js, jiang}. Meanwhile, variations of (1) when dealing with $M=L^{\infty}(X)\rtimes G$ have also been used in \cite{packer, chifan_das, js}. By contrast, it seems the use of (2) has not received much attention besides in \cite{haga,cs_internat, js, jiang}. 
Our strategy is to first locate certain $u_g^*E(u_g)$ inside a small enough von Neumann subalgebra by (1) and then use (2) to get sufficiently many hidden relations among these unknowns to completely solve for them. Similar idea has been applied in proving \cite[Proposition 5.6]{jiang}. 

Besides the use of (2) above, we also need two more ingredients for the proof.

First, we will first study the inclusion $L(SL_2(\mathbb{Q}))\subset L(\mathbb{Q}^2\rtimes SL_2(\mathbb{Q}))$ rather than the $\mathbb{Z}$-coefficient inclusion. The reason is that the affine action $\mathbb{\textbf{k}}^2\rtimes SL_2(\textbf{k})\curvearrowright \textbf{k}^2$ is 2-transitive (equivalently, $SL_2(\textbf{k})\curvearrowright \textbf{k}^2\setminus\{(0, 0)\}$ is transitive, see \cite[Def. 4.2]{jiang}) for $\textbf{k}=\mathbb{Q}$ but not for $\textbf{k}=\mathbb{Z}$. This will make the calculation while trying to solve for the unknowns much easier for $\textbf{k}=\mathbb{Q}$. Consequently, the analogue of the above theorem for $\Q$-coefficient also holds true, see Corollary \ref{cor: L(SL_2(Q)) is mHAP}.

Second, for both coefficients, i.e. $\textbf{k}=\mathbb{Q}$ or $\mathbb{Z}$, we are not able to determine all intermediate von Neumann subalgebras in the ambient algebra $L(\textbf{k}^2\rtimes SL_2(\textbf{k}))$ directly by the above strategy. Instead, we show it works perfectly after restricting to a ``large" von Neumann subalgebra (denoted by $M_0$), i.e. the crossed product coming from the quotient action of the algebraic action $SL_2(\textbf{k})\curvearrowright \widehat{\textbf{k}^2}$ by modding out the relation $\phi\sim \phi'$, where $\phi$, $\phi'\in \widehat{\textbf{k}^2}$ and $\phi'(x, y):=\phi(-x, -y)$ for all $(x, y)\in \textbf{k}^2$. Although the necessity of taking this restriction is unclear to us, it does help solving for the unknowns.

Note that the abovementioned subalgebra $M_0$ has Pimsner-Popa index \cite{pimsner_popa} two inside $L(\textbf{k}^2\rtimes SL_2(\textbf{k}))$. The key result we get is a complete description of all von Neumann subalgebras containing $L(SL_2(\textbf{k}))$ while sitting in $M_0$, i.e. Theorem \ref{thm: complete description of intermediate vn algs} (for $\textbf{k}=\Q$) and Theorem \ref{thm: SL_2(Z)-thm} (for $\textbf{k}=\Z$).  With this description at hand, Theorem \ref{main thm} can be proved using a standard argument based on the works \cite{chifan_das, ioana_duke, jones_xu}, see the proof of Corollary \ref{cor: L(SL_2(Q)) is mHAP}, Corollary \ref{cor: L(SL_2(Z)) is mHAP}.
 
In view of \cite[Proposition 2.19]{js}, it might be interesting to study whether the group von Neumann algebra of upper triangular matrices in $SL_3(\Z)$ is maximal Haagerup in $L(SL_3(\Z))$.\\

\paragraph*{\textbf{Organization of the paper:}} In Section \ref{section: preliminaries}, we briefly review several notions, including Haagerup property, relative property (T), algebraic actions, weak mixing and compactness. We study the $\Q$-coefficient inclusion $L(SL_2(\Q))\subset  L(\Q^2\rtimes SL_2(\Q))$ in Section \ref{section: Q-case}, which is split into two subsections. In subsection \ref{subsection: Q-case, factor}, we consider the easier case: inclusion of factors $L(PSL_2(\Q))\subset L^{\infty}(Y)\rtimes PSL_2(\Q)$ and show no other non-trivial intermediate von Neumann subalgebras  exist (Theorem \ref{thm: no intermediate subalgs}). Then all intermediate von Neumann subalgebras between $L(SL_2(\Q))$ and $L^{\infty}(Y)\rtimes SL_2(\Q)$ are determined (Theorem \ref{thm: complete description of intermediate vn algs}) in subsection \ref{subsection: Q-case, non-factor}. Moreover, we deduce three corollaries (Corollary \ref{cor: weakly mixing prime action}, \ref{cor: max Haagerup for LG}, \ref{cor: L(SL_2(Q)) is mHAP}) in this section. These corollaries show that $PSL_2(\Q)\curvearrowright Y$ is a free, weakly mixing and prime action and $L(PSL_2(\Q))$ (resp. $L(SL_2(\Q))$) is maximal Haagerup inside $L^{\infty}(Y)\rtimes PSL_2(\Q)$ (resp. $L(\Q^2\rtimes SL_2(\Q))$). In Section \ref{section: Z-case}, we show how to modify the proof in previous section to deal with the $\Z$-coefficient inclusion $L(SL_2(\Z))\subset L(\Z^2\rtimes SL_2(\Z))$. Theorem \ref{main thm} is proved in this section. Then the paper is finished with an appendix, where we include details for the induction step in the proof of Theorem \ref{thm: SL_2(Z)-thm}, which describes all intermediate von Neumann subalgebras between $L(SL_2(\Z))$ and $L^{\infty}(Y)\rtimes SL_2(\Z)$.\\

\paragraph*{\textbf{Notations:}} The following notations will be used in the context.
\begin{itemize}
\item For a p.m.p. action $G\curvearrowright (X, \mu)$, $ker(G\curvearrowright X)$ denotes the kernel of the action, i.e. $ker(G\curvearrowright X):=\{g\in G: gx=x, ~\forall~\mu-a.e.~x\in X\}$.
\item Let $N\subset M$ be an inclusion of von Neumann algebras. If $\mathcal{A}\subset Aut(M)$ is a set, then $N^{\mathcal{A}}:=\{x\in N: \alpha(x)=x,~ \forall~\alpha\in \mathcal{A}\}$. If $\mathcal{A}=\{\phi\}$, we also write $Fix(\phi)$ for $M^{\mathcal{A}}$.
\end{itemize}
\section{Preliminaries}\label{section: preliminaries}
\subsection{Haagerup property v.s. relative property (T)} Haagerup property originated in the work of Haagerup on free groups \cite{haagerup}. Later on, this approximation property was proved to be very fruitful for the study of operator algebras. In particular, this property was defined for finite von Neumann algebras in \cite{choda_m, jolissaint}. We will not use its definition directly in this paper, for which we refer to \cite{ccjjv}, \cite[Section 1]{js}. Instead, let us recall that a key obstacle for the Haagerup property in both the group setting and von Neumann algebra setting is the relative property (T) \cite{bdv, connes_jones, popa_annals}. More precisely, we will frequently use two standard facts: (1) If a group $G$ contains an infinite subgroup with relative property (T), then it does not have Haagerup property. For example, $\Z^2\rtimes SL_2(\Z)$ does not have Haagerup property as $\Z^2$ is a subgroup with relative property (T) \cite{margulis}. (2) If $N$ is a diffuse von Neumann subalgebra inside a finite von Neumann algebra $M$ and $N\subset M$ has relative property (T), then $M$ does not have Haagerup property \cite{popa_annals}.
\subsection{Algebraic actions}
Let $(X, \mu)$ be a compact metrizable abelian group equipped with the Haar measure $\mu$ and $\alpha: G\to Aut(X)$ be a group homomorphism from a countable discrete group $G$ to the continuous automorphism group $Aut(X)$. Then $\alpha: G\curvearrowright (X, \mu)$ is called an \emph{algebraic action}. Notice that the Pontryagin dual $\widehat{X}$ inherits a $G$-module structure. Conversely, given a countable $\Z G$-module $M$, it induces an algebraic action $G\curvearrowright \widehat{M}$ defined by $\langle g\chi, m\rangle:=\langle \chi, g^{-1}m\rangle$ for all $g\in G$, $\chi\in \widehat{M}$ and all $m\in M$, where $\langle \cdot, \cdot\rangle$ denotes the pairing between $\widehat{M}$ (the Pontryagin dual of $M$) and $M$. In this paper, we work with the algebraic action $SL_2(\textbf{k})\curvearrowright\widehat{\textbf{k}^2}$ for $\textbf{k}=\Q$ and $\Z$, where $\textbf{k}^2$ is treated as an $SL_2(\textbf{k})$-module defined by matrix multiplication from the left. A basic fact we frequently use is that we have an isomorphism $L(M\rtimes G)\cong L^{\infty}(\widehat{M})\rtimes G$ for an algebraic action $G\curvearrowright \widehat{M}$. For more discussion on algebraic actions, see \cite{schmidt_book, kerrli_book}.
\subsection{Weak mixing v.s. compactness} Let $G\curvearrowright (X,\mu)$ be a p.m.p. (probability-measure preserving) action. Recall that it is called \emph{weakly mixing} if for every finite collection $\Omega$ of measurable subsets of $X$ and every $\epsilon>0$ there exists an $s\in G$ such that $|\mu(sA\cap B)-\mu(A)\mu(B)|<\epsilon$ for all $A$, $B\in \Omega$. Several conditions are known to be equivalent to being weakly mixing \cite[Theorem 2.25]{kerrli_book}, one of which is to require the only compact elements in $L^2(X)$ under the Koopman representation (i.e. the unitary representation $G\curvearrowright L^2(X)$ defined by $(sf)(x):=f(s^{-1}x)$ for all $s\in G$ and a.e. $x\in X$) are the a.e. constant functions. Here, recall that for a unitary representation $\pi: G\curvearrowright \mathcal{H}$, an element $\xi\in \mathcal{H}$ is called \emph{compact} if the set $\overline{\pi(G)\xi}$ is compact. For a p.m.p. action $G\curvearrowright (X, \mu)$, it is called \emph{compact} if its Koopman representation $\pi$ is compact, i.e. for every $\xi\in L^2(X)$, $\xi$ is compact. Clearly, for a non-trivial p.m.p. action $G\curvearrowright (X, \mu)$, if it is weakly mixing, then it is not compact and every non-trivial quotient action is still weakly mixing.


\section{Complete description of intermediate von Neumann subalgebras: $\Q$-coefficient}\label{section: Q-case}

In this section, we work with the $\mathbb{Q}$-coefficient inclusion: $L(SL_2(\mathbb{Q}))\subset L(\mathbb{Q}^2\rtimes SL_2(\mathbb{Q}))$.

Let $G=SL_2(\mathbb{Q})$ and $\bar{G}=PSL_2(\mathbb{Q})$. Let $G\curvearrowright X$ be the algebraic action $G\curvearrowright \widehat{\mathbb{Q}^2}$. Consider the quotient action $G\curvearrowright Y$, where $Y$ is defined by modding out the relation $\phi\sim \phi'$, where $\phi$, $\phi'\in X=\widehat{\mathbb{Q}^2}$ and $\phi'(x, y):=\phi(-x,-y)$ for all $(x, y)\in \mathbb{Q}^2$.
In other words, $L^{\infty}(Y)\rtimes G\cong A\rtimes G$, where $A$ denotes the von Neumann subalgebra of $L(\mathbb{Q}^2)$ consisting of all elements $\sum_{x, y}c_{x, y}u_{x, y}$ such that $c_{x, y}=c_{-x, -y}$ for all $(x, y)\in \mathbb{Q}^2$.
Note that $-id\in ker(G\curvearrowright Y)$, where $id $ stands for the identity matrix in $G$. Therefore, $G\curvearrowright Y$ descends to an action $\bar{G}\curvearrowright Y$. Note that these notations will be used throughout this section unless otherwise stated.

The main result in this section is the following theorem.

\begin{theorem}\label{thm: complete description of intermediate vn algs}
With the above notations, let $P$ be any intermediate von Neumann subalgebra between $L(G)$ and $A\rtimes G$, then 
\[P\in \bigg\{L(G), A\rtimes G, qL(G)\oplus (1-q)(A\rtimes G), (1-q)L(G)\oplus q(A\rtimes G)\bigg\},\] where $q=\frac{u_{id}+u_{-id}}{2}$ is a central projection in $A\rtimes G$. 
\end{theorem}

To prove this, we need to first study $\bar{G}=PSL_2(\mathbb{Q})\curvearrowright Y$.


\begin{theorem}\label{thm: no intermediate subalgs}
Let $P$ be an intermediate von Neumann subalgebra between $L(\bar{G})$ and $L^{\infty}(Y)\rtimes \bar{G}$. Then $P=L(\bar{G})$ or $L^{\infty}(Y)\rtimes \bar{G}$. In other words, $L(\bar{G})$ is a maximal von Neumann subalgebra in $L^{\infty}(Y)\rtimes \bar{G}$.
\end{theorem}

To the best of our knowledge, Theorem \ref{thm: no intermediate subalgs} gives the first concrete inclusion of von Neumann algebras of the form $LH\subset L^{\infty}(X)\rtimes H$ such that $H\curvearrowright X$ is a free, weakly mixing p.m.p. action and all intermediate von Neumann subalgebras can be described. Note that several similar results have appeared recently while assuming $H\curvearrowright X$ is a non-faithful action \cite{amrutam, js}, profinite action \cite{js, chifan_das}, or more generally a compact action \cite{chifan_das}.

To present a direct corollary, which may be of independent interest, we first recall that for a p.m.p. action, it is called \emph{prime} if it admits no non-trivial quotient actions. We are unaware of any other concrete free prime actions of any non-amenable groups besides the one in the following corollary.

\begin{corollary}\label{cor: weakly mixing prime action}
With the above notations, $\bar{G}\curvearrowright Y$ is a free, weakly mixing and prime action.
\end{corollary}
\begin{proof}
Clearly, any non-trivial quotient action gives rise to a non-trivial intermediate von Neumann subalgebra. Theorem \ref{thm: no intermediate subalgs} implies the action is prime. Freeness part is easy to check and we leave it as an exercise.
We are left to check $\bar{G}\curvearrowright Y$ is weakly mixing. Equivalently, we need to check $G\curvearrowright Y$ is weakly mixing. As $G\curvearrowright Y$ is a quotient action of $G\curvearrowright X$, it suffices to show $G\curvearrowright X$ is weakly mixing. Since $G\curvearrowright X$ is an algebraic action, by \cite[Proposition 2.36]{kerrli_book}, we just need to check that $[G: Stab(a)]=\infty$ for each $0\neq a\in \widehat{X}=\mathbb{Q}^2$, where $Stab(a):=\{g\in G: g.a=a\}$. If $a=e_1:=\big(\begin{smallmatrix}1\\0 \end{smallmatrix}\big)$, then $Stab(e_1)=\big(\begin{smallmatrix} 1&\mathbb{Q}\\
0&1\end{smallmatrix}\big)$. Clearly, $[G: Stab(e_1)]=\infty$. For a general $0\neq a\in\mathbb{Q}^2$, there exists some $g\in SL_2(\mathbb{Q})$ such that $a=g.e_1$. Therefore, $[G: Stab(a)]=[G: gStab(e_1)g^{-1}]=[G: Stab(e_1)]=\infty$.
\end{proof}

Moreover, we also have the following two corollaries, whose proofs will be given at the end of subsection \ref{subsection: Q-case, factor} and subsection \ref{subsection: Q-case, non-factor} respectively.

\begin{corollary}\label{cor: max Haagerup for LG}
With the above notations, $L(\bar{G})$ is a maximal Haagerup von Neumann subalgebra in $L^{\infty}(Y)\rtimes \bar{G}$.
\end{corollary}

\begin{corollary}\label{cor: L(SL_2(Q)) is mHAP}
With the above notations, $LG$ is maximal Haagerup inside $L(\Q^2\rtimes G)$. 
\end{corollary}

\subsection{Factor inclusion: $L(PSL_2(\mathbb{Q}))\subset L^{\infty}(Y)\rtimes PSL_2(\mathbb{Q})$}\label{subsection: Q-case, factor}

In this subsection, we prove Theorem \ref{thm: no intermediate subalgs} and Corollary \ref{cor: max Haagerup for LG}. 
\begin{proof}[Proof of Theorem \ref{thm: no intermediate subalgs}]
The proof is inspired by the proof of \cite[Proposition 5.7]{jiang}. We split the proof into several steps.

\textbf{Step 1}: preparation and setting up notations.

Let $P$ be any von Neumann subalgebra between $N:=L(\bar{G})$ and $M:=L^{\infty}(Y)\rtimes \bar{G}$. Let $A:=L^{\infty}(Y)\subset L(\mathbb{Q}^2)$ and $E$ be the trace preserving conditional expectation onto $P$. 

Let $e_1=\big(\begin{smallmatrix}1\\0\end{smallmatrix}\big)$ and $e_2=\big(\begin{smallmatrix}0\\1\end{smallmatrix}\big)$. For any $g\in G$ and $v\in \mathbb{Q}^2$, we write $g.v$ for the matrix multiplication between $g$ and $v$. Let $K=\big(\begin{smallmatrix}1&\mathbb{Q}\\
0&1\end{smallmatrix}\big)$ and $\bar{K}$ be the image of $K$ under the quotient map $G\twoheadrightarrow \bar{G}$. 

For any $0\neq x\in \mathbb{Q}$ and $y\in\mathbb{Q}$, define $p_{x, y}=\big(\begin{smallmatrix}
x&0\\
y&\frac{1}{x}
\end{smallmatrix}\big)$, $q_x=\big(\begin{smallmatrix}
0&\frac{-1}{x}\\
x&0
\end{smallmatrix}\big)$ and $r_x=\big(\begin{smallmatrix}
1&x\\
0&1
\end{smallmatrix}\big)$. Clearly, $p_{x, y}^{-1}=p_{\frac{1}{x}, -y}$, $q_x^{-1}=q_{-x}$ and $r_x^{-1}=r_{-x}.$

Observe that if $\phi\in Aut(L^{\infty}(Y)\rtimes \bar{G})$ and $\phi(P)=P$, then $\phi(E(u_{e_1}+u_{-e_1}))=E(\phi(u_{e_1}+u_{-e_1}))$ as $E(u_{e_1}+u_{-e_1})\in P$. In particular, take $\phi=Ad(u_g)$ for any $g\in \bar{G}$, we get 
\begin{align}\label{eq: compute E(u_b+u_-b) by E(u_e_1+u_-e_1)}
    u_gE(u_{e_1}+u_{-e_1})u_g^*=E(u_{g.{e_1}}+u_{-g.{e_1}}).
\end{align}
Clearly, if $g\in K$, then $g.{e_1}=e_1$. Therefore, we deduce that 
\begin{align*}
    E(u_{e_1}+u_{-e_1})\in L(\bar{K})'\cap M.
\end{align*}

Claim 1: $L(\bar{K})'\cap M\subseteq (A\cap L(e_1\mathbb{Q}))\rtimes \bar{K}$, where $\mathbb{Q}e_1=\big(\begin{smallmatrix}
\mathbb{Q}\\0
\end{smallmatrix}\big)\subset\mathbb{Q}^2$.
\begin{proof}[Proof of Claim 1]
Let $a\in L(\bar{K})'\cap M$ and $a=\sum_{vg}\lambda_{vg}vg$ be its Fourier expansion, where $\lambda_{vg}=\lambda_{(-v)g}\in\mathbb{C}$ for all $v\in \mathbb{Q}^2$ and all $g\in \bar{G}$.
Then, $r_{nx}ar_{-nx}=a$ for every $0\neq x\in\mathbb{Q}$ and every integer $n$. Hence, we get
$    \lambda_{vg}=\lambda_{(r_{nx}.v)(r_{nx}gr_{-nx})}$ for all $v\in \mathbb{Q}^2$, all $x\in \Q\setminus\{0\}$ and all $g\in \bar{G}$.
Now, one can check that if $g\not\in\bar{K}$ and $x\neq 0$, then 
$r_{nx}gr_{-nx}\neq r_{mx}gr_{-mx}$ for all $n\neq m$.
Therefore, we must have $\lambda_{vg}=0$ if $v\in \mathbb{Q}^2$ and $g\not\in \bar{K}$.

Let $g\in\bar{K}$. If $bx\neq 0$, then clearly  $n\neq m$ implies
$r_{nx}.v\neq r_{mx}.v$, where $v=\big(\begin{smallmatrix}
c\\b
\end{smallmatrix}\big)$,
which in turn implies that $\lambda_{vg}=0$ if $v\not\in \mathbb{Q}e_1$. Therefore, $a\in (A\cap L(\mathbb{Q}e_1))\rtimes \bar{K}$.
\end{proof}
By Claim 1, we can write 
   $ E(u_{e_1}+u_{-e_1})=\sum_{x, y\in \mathbb{Q}}\lambda_{x, y}\big(\begin{smallmatrix}x\\0\end{smallmatrix}\big)\big(\begin{smallmatrix}1&y\\0&1\end{smallmatrix}\big),$
where $\lambda_{x, y}=\lambda_{-x,y}$ for all $(x, y)\in\mathbb{Q}^2$.
In other words, 
\begin{align}\label{def: def of E(e_1+-e_1)}
E(u_{e_1}+u_{-e_1})=\frac{1}{2}\sum_{x, y\in \mathbb{Q}}\lambda_{x, y}\big[\big(\begin{smallmatrix}x\\0\end{smallmatrix}\big)+\big(\begin{smallmatrix}-x\\0\end{smallmatrix}\big)\big]\big(\begin{smallmatrix}1&y\\0&1\end{smallmatrix}\big).   
\end{align}
\begin{remark}
From now on, we will not mention the range for $(x, y)$ if no confusion arises. Moreover, we simplify the notation by using $v$ to denote the canonical unitary $u_v$ for all $v\in \mathbb{Q}^2$. So be alert that $v_1+v_2$ stands for $u_{v_1}+u_{v_2}$ instead of the sum of two vectors in $\mathbb{Q}^2$.  We also remind the reader that all $2\times 2$ matrices in the proof of Theorem \ref{thm: no intermediate subalgs} belong to $\bar{G}$. In other words, we identify elements in $G$ with their images in $\bar{G}$ to simplify notations.
\end{remark}
\textbf{Step 2}: find restrictions on $\lambda_{x, y}$.

First, for any $b\in \mathbb{Q}$, let $g=\big(\begin{smallmatrix}
1&b\\0&1
\end{smallmatrix}\big)$. Since $u_g\in P$, we deduce 
\begin{align}\label{eq: 2nd restriction}
\lambda_{0, b}=\langle E(u_{e_1}+u_{-e_1}),u_g \rangle=\langle u_{e_1}+u_{-e_1}, u_g\rangle=0.
\end{align}

Next, as $E$ satisfies $P$-bimodular property, we get
\begin{align}\label{eq: E(E() )=E()E()}
E(E(u_{e_1}+u_{-e_1})(u_{e_2}+u_{-e_2}))=E(u_{e_1}+u_{-e_1})E(u_{e_2}+u_{-e_2}).
\end{align}
Set $g=\big(\begin{smallmatrix}0&-1\\1&0\end{smallmatrix}\big)\in \bar{G}$ in (\ref{eq: compute E(u_b+u_-b) by E(u_e_1+u_-e_1)}) and note that $e_2=g.e_1$, we get that
\begin{align}\label{def: def of E(e_2+-e_2)}
\begin{split} 
E(u_{e_2}+u_{-e_2})&=u_gE(u_{e_1}+u_{-e_1})u_g^*\overset{(\ref{def: def of E(e_1+-e_1)})}{=}\sum_{x, y}\lambda_{x, y}\big(\begin{smallmatrix}0\\x \end{smallmatrix}\big)\big(\begin{smallmatrix} 1&0\\-y&1\end{smallmatrix}\big)\\
&=\frac{1}{2}\sum_{x, y}\lambda_{x, y}\big[\big(\begin{smallmatrix}0\\x \end{smallmatrix}\big)+\big(\begin{smallmatrix}0\\-x \end{smallmatrix}\big)\big]\big(\begin{smallmatrix} 1&0\\-y&1\end{smallmatrix}\big).
\end{split}
\end{align}

Next, let us compute both sides of (\ref{eq: E(E() )=E()E()}). 

By (\ref{def: def of E(e_1+-e_1)}), a simple calculation shows that the LHS of (\ref{eq: E(E() )=E()E()}) equals  
\[\frac{1}{2}\sum_{x, y}\lambda_{x, y}E\big[ \big(\begin{smallmatrix}x+y\\1 \end{smallmatrix}\big) +\big(\begin{smallmatrix}-x-y\\-1 \end{smallmatrix}\big)  \big]\big(\begin{smallmatrix} 1&y\\0&1\end{smallmatrix}\big)
    +\frac{1}{2}\sum_{x, y}\lambda_{x, y}E\big[ \big(\begin{smallmatrix}-x+y\\1 \end{smallmatrix}\big)+\big(\begin{smallmatrix}x-y\\-1 \end{smallmatrix}\big) \big]\big(\begin{smallmatrix} 1&y\\0&1\end{smallmatrix}\big).\]

Using $\lambda_{x, y}=\lambda_{-x, y}$ and doing change of variables, we deduce that the two summands above are equal. Besides, we can express both $E\big[ \big(\begin{smallmatrix}x+y\\1 \end{smallmatrix}\big) +\big(\begin{smallmatrix}-x-y\\-1 \end{smallmatrix}\big)  \big]$  and $E\big[ \big(\begin{smallmatrix}-x+y\\1 \end{smallmatrix}\big)+\big(\begin{smallmatrix}x-y\\-1 \end{smallmatrix}\big) \big]$ in terms of $E(u_{e_1}+u_{-e_1})$ by setting $v=e_1$, $g=\big(\begin{smallmatrix} x+y&x+y-1\\
 1&1\end{smallmatrix}\big)$ and $g=\big(\begin{smallmatrix} -x+y&-x+y-1\\
 1&1\end{smallmatrix}\big)$ respectively in (\ref{eq: compute E(u_b+u_-b) by E(u_e_1+u_-e_1)}).
  
Based on the above facts and (\ref{def: def of E(e_1+-e_1)}), we can continue the above calculation to get:
\begin{align*}
    \quad\mbox{LHS of (\ref{eq: E(E() )=E()E()})}&=\sum_{x, y, a, b}\lambda_{x,y}\lambda_{a, b}\big(\begin{smallmatrix}
a(x+y)\\a
\end{smallmatrix}\big)\big(\begin{smallmatrix}
1-b(x+y)&y+bx(x+y)\\
-b&1+bx
\end{smallmatrix}\big)\\
&=\frac{1}{2}\sum_{x, y, a, b}\lambda_{x,y}\lambda_{a, b}\big[\big(\begin{smallmatrix}
a(x+y)\\a
\end{smallmatrix}\big)+\big(\begin{smallmatrix}
-a(x+y)\\-a
\end{smallmatrix}\big) \big]\big(\begin{smallmatrix}
1-b(x+y)&y+bx(x+y)\\
-b&1+bx
\end{smallmatrix}\big)\\
&\mbox{(by $\lambda_{a, b}=\lambda_{-a, b}$ and change of variables)}\\
&=\sum_{a>0}\sum_{x, y, b}\lambda_{x, y}\lambda_{a, b}\big[\big(\begin{smallmatrix}
a(x+y)\\a
\end{smallmatrix}\big)+\big(\begin{smallmatrix}
-a(x+y)\\-a
\end{smallmatrix}\big)\big]\big(\begin{smallmatrix}
1-b(x+y)&y+bx(x+y)\\
-b&1+bx
\end{smallmatrix}\big). 
 \end{align*}
Here, to get the last equality, we have used (\ref{eq: 2nd restriction}) to cross out the terms corresponding to $a=0$.

Then, using (\ref{def: def of E(e_1+-e_1)}) and (\ref{def: def of E(e_2+-e_2)}), we get
\begin{align*} 
\mbox{RHS of (\ref{eq: E(E() )=E()E()})}
&=\sum_{x, y}\sum_{a, b}\frac{\lambda_{x, y}\lambda_{a, b}}{2}\big[ \big(\begin{smallmatrix}
x+ay\\a
\end{smallmatrix}\big)+\big(\begin{smallmatrix}
-x-ay\\-a
\end{smallmatrix}\big)\big]\big(\begin{smallmatrix}
1-by&y\\
-b&1
\end{smallmatrix}\big)\\
&{(\mbox{use $\lambda_{x, y}=\lambda_{-x, y}$ and $\lambda_{a, b}=\lambda_{-a, b}$.})}\\
&=\sum_{x, y, b}\sum_{a>0}\lambda_{x, y}\lambda_{a, b}\big[ \big(\begin{smallmatrix}
x+ay\\a
\end{smallmatrix}\big)+\big(\begin{smallmatrix}
-x-ay\\-a
\end{smallmatrix}\big)\big]\big(\begin{smallmatrix}
1-by&y\\
-b&1
\end{smallmatrix}\big).\\
&{(\mbox{use $\lambda_{x, y}=\lambda_{-x, y}$ and $\lambda_{a, b}=\lambda_{-a, b}$ and do change of variables.})}
\end{align*}
Once again, we cross out two terms corresponding to $a=0$ by (\ref{eq: 2nd restriction}) to get the last equality.

Now, fix any $a>0$ and $(x, y, b)\in \mathbb{Q}^3$. By comparing the coefficients in front of the term 
$\big[ \big(\begin{smallmatrix}
x+ay\\a
\end{smallmatrix}\big)+\big(\begin{smallmatrix}
-x-ay\\-a
\end{smallmatrix}\big) \big]\big(\begin{smallmatrix}
1-by&y\\
-b&1
\end{smallmatrix}\big)$ for both sides of (\ref{eq: E(E() )=E()E()}), we deduce that for any $a>0$ and $(x, y, b)\in \mathbb{Q}^3$, 
\begin{align}\label{eq: eq1-6}
\lambda_{x, y}\lambda_{a, b}=
\begin{cases}
\lambda_{\frac{x}{a}, y}\lambda_{a, b}, &~\mbox{if $bx=0$}\\
\lambda_{\frac{2}{b},y-\frac{4}{b}}\lambda_{a, -b}, &~\mbox{if~ $2a+bx=0$}\\
0, &~\mbox{otherwise}.
\end{cases}
\end{align} 

We remind the reader that the three cases appear due to the fact that the acting group in the context of Theorem \ref{thm: no intermediate subalgs} is $\bar{G}$ instead of $G$.

\textbf{Step 3}: use restrictions to deduce that $P=N$ or $M$.

Take $(a, b)=(x,y)$ in the last case in (\ref{eq: eq1-6}), we deduce that $\lambda_{x, y}=0$. Note that the assumptions in this case, i.e. $a>0$, $bx\neq 0$ and $2a+bx\neq 0$ are equivalent to $x>0$ and $y\neq 0$, $-2$ (w.r.t. our choice $(a, b)=(x, y)$). 
So, $\lambda_{x, y}=0$ for all $x\neq 0$ and $y\not\in\{0, -2\}$ as $\lambda_{x, y}=\lambda_{-x, y}$ holds for all $(x, y)\in \mathbb{Q}^2$.

Take $b=y=-2$ and $a=x>0$ in the 2nd case of (\ref{eq: eq1-6}), we deduce that $\lambda_{x, -2}^2=\lambda_{-1, 0}\lambda_{x, 2}=0$ as $\lambda_{x, 2}=0$ from above. Then $\lambda_{x, -2}=0$. Note that the assumptions in this case, i.e. $a>0$ and $2a+bx=0$ are reduced to $x>0$ (w.r.t. our choice $b=y=-2$ and $a=x>0$). So $\lambda_{x, -2}=0$ for all $x\neq 0$.

To sum up, we have proved that $\lambda_{x, y}=0$ for all $x\neq 0$ and $y\neq 0$. Recall that $\lambda_{0, b}=0$ for all $b\in \mathbb{Q}$ by (\ref{eq: 2nd restriction}). Hence $\lambda_{x, y}=0$ unless $y=0$.  

Now, we claim that $\lambda_{a, 0}=0$ for all $a\not\in \{0, \pm 1\}$.

Assume not, then there exist some $a\not\in \{0, \pm 1\}$ such that $\lambda_{a,0}\neq 0$. We may further assume $a>0$. Set $b=y=0$ in the 1st case of (\ref{eq: eq1-6}), we get that $\lambda_{x, 0}=\lambda_{\frac{x}{a}, 0}$ for all $x\neq 0$; equivalently, $\lambda_{ax,0}=\lambda_{x, 0}$ for all $x\neq 0$.
Therefore, $\lambda_{x, 0}=\lambda_{ax, 0}=\lambda_{a^2x, 0}=\cdots=\lambda_{a^nx,0}$ for all $n\geq 1$ if $x\neq 0$. But $a^nx\neq a^mx$ for all $n\neq m$ as $a\in \mathbb{Q}\setminus\{0, \pm 1\}$, this implies that $\lambda_{x, 0}=0$ for all $x\neq 0$, which contradicts to the assumption that $\lambda_{a, 0}\neq 0$.

Therefore, we have shown that $\lambda_{x, y}=0$ unless $(x, y)=(\pm 1, 0)$. Hence, $E(u_{e_1}+u_{-e_1})=\lambda_{1, 0}\big(\begin{smallmatrix}1\\0\end{smallmatrix}\big)+\lambda_{-1, 0}\big(\begin{smallmatrix} -1\\0\end{smallmatrix}\big)=\lambda_{1, 0}(u_{e_1}+u_{-e_1})$. Applying $E$ on both sides, we deduce that $\lambda_{1, 0}=\lambda_{1, 0}^2$, i.e. $\lambda_{1, 0}=0$ or $1$. This implies that $P=N$ or $M$.
\end{proof}
Next, let us prove Corollary \ref{cor: max Haagerup for LG}.
\begin{proof}[Proof of Corollary \ref{cor: max Haagerup for LG}]
First, $SL_2(\mathbb{Q})$ has Haagerup property as a discrete group by \cite{ghw}, therefore, $L(SL_2(\mathbb{Q}))$ has Haagerup property by \cite{choda_m}. Let $q=\frac{u_{id}+u_{-id}}{2}$. As $qL(SL_2(\mathbb{Q}))q\cong L(PSL_2(\mathbb{Q}))$ via the isomorphism $\pi$ as defined in next paragraph, we know $L(PSL_2(\mathbb{Q}))$ also has Haagerup property by \cite[Proposition 2.4.1]{popa_annals}. Hence $\bar{G}=PSL_2(\mathbb{Q})$ also has Haagerup property by \cite{choda_m}. By Theorem \ref{thm: no intermediate subalgs}, we are left to show $L^{\infty}(Y)\rtimes \bar{G}$ does not have Haagerup property. Observe that we have a natural embedding $L^{\infty}(Z)\rtimes PSL_2(\mathbb{Z})\hookrightarrow L^{\infty}(Y)\rtimes \bar{G}$ induced by the inclusion $\Z\hookrightarrow \Q$, where $L^{\infty}(Z)$ is defined similarly as $L^{\infty}(Y)$ but using $\Z$-coefficient, i.e. $L^{\infty}(Z)\rtimes PSL_2(\mathbb{Z})=L^{\infty}(\mathbb{T}^2/{\sim})\rtimes PSL_2(\mathbb{Z})$. It suffices to prove that $L^{\infty}(\mathbb{T}^2/{\sim})\subset L^{\infty}(\mathbb{T}^2/{\sim})\rtimes PSL_2(\mathbb{Z})$ has relative property (T).

First, by \cite[Lemma 3.5]{js}, we know that  $L^{\infty}(\mathbb{T}^2/{\sim})\subset L^{\infty}(\mathbb{T}^2/{\sim})\rtimes SL_2(\mathbb{Z})$ has relative property (T). Then, observe that we have a $*$-homomorphism $\phi: L^{\infty}(\mathbb{T}^2/{\sim})\rtimes SL_2(\mathbb{Z})\to L^{\infty}(\mathbb{T}^2/{\sim})\rtimes PSL_2(\mathbb{Z})$ induced by the quotient map $SL_2(\mathbb{Z})\twoheadrightarrow PSL_2(\mathbb{Z})$. Indeed, let $q=\frac{u_{id}+u_{-id}}{2}$. We define $\phi$ as the composition of the following maps:
\begin{align*}
&\phi: L^{\infty}(\mathbb{T}^2/{\sim})\rtimes SL_2(\mathbb{Z})\overset{q\cdot q}{\twoheadrightarrow}q(L^{\infty}(\mathbb{T}^2/{\sim})\rtimes SL_2(\mathbb{Z}))q\overset{\pi}{\cong}L^{\infty}(\mathbb{T}^2/{\sim})\rtimes PSL_2(\mathbb{Z})\\
&\quad\quad au_g\overset{q\cdot q}{\longmapsto} q(au_g)q\overset{\pi}{\longmapsto} au_{\bar{g}}.
\end{align*}
To check $\phi$ is a well-defined $*$-homomorphism, it suffices to show that the above map $q\cdot q$ is a $*$-homomorphism and the map $\pi$ is an isomorphism.

The first part holds since $q$ lies in the center of $L^{\infty}(\mathbb{T}^2/{\sim})\rtimes SL_2(\mathbb{Z})$. To check the second part holds, one defines a unitary $U: q\ell^2(SL_2(\Z))\cong \ell^2(PSL_2(\Z))$ by $U(q\delta_g)=\delta_{\bar{g}}/\sqrt{2}$. Let us check $U$ is well-defined. Assume $q(\sum_g\lambda_g\delta_g)=0$, i.e. $\sum_g\lambda_g\frac{\delta_g+\delta_{-g}}{2}=0$, or equivalently, $\sum_g\frac{\lambda_g+\lambda_{-g}}{2}\delta_g=0$, we need to show $\sum_g\lambda_g\delta_{\bar{g}}=0$. Clearly, this holds iff $\lambda_g+\lambda_{-g}=0$ for all $g\in SL_2(\Z)$ iff $q(\sum_g\lambda_g\delta_g)=0$. Then it is easy to check that $\pi=Ad(\hat{U})|_{q(L^{\infty}(\mathbb{T}^2/{\sim})\rtimes SL_2(\mathbb{Z}))q}$, where $\hat{U}: q(L^2({\mathbb{T}^2/{\sim}}))\bar{\otimes}\ell^2(SL_2(\Z)))\cong L^2(\mathbb{T}^2/{\sim})\bar{\otimes}\ell^2(PSL_2(\Z))$ is the unitary operator $Id\otimes U$, i.e. $\hat{U}(q(\xi\otimes\delta_g))=\xi\otimes \delta_{\bar{g}}/\sqrt{2}$ for all $\xi\in L^2(\mathbb{T}^2/{\sim})$ and all $g\in SL_2(\Z)$.

Notice that $\phi|_{L^{\infty}(\mathbb{T}^2/{\sim})}=id$, a standard argument using $\phi$ shows that relative property (T) for the inclusion $L^{\infty}(\mathbb{T}^2/{\sim})\subset L^{\infty}(\mathbb{T}^2/{\sim})\rtimes SL_2(\mathbb{Z})$ transfers to the inclusion $L^{\infty}(\mathbb{T}^2/{\sim})\subset L^{\infty}(\mathbb{T}^2/{\sim})\rtimes PSL_2(\mathbb{Z})$.
\end{proof}

\subsection{Non-factor inclusion: $L(SL_2(\mathbb{Q}))\subset                                                                                                                                                                                                                                                                                                                                                                                                                                                                                                                                                                                                                                                                                                                                                                                                                                                                                            L^{\infty}(Y)\rtimes SL_2(\mathbb{Q})$}\label{subsection: Q-case, non-factor}

In this subsection, we show how to use results in the previous subsection to prove Theorem \ref{thm: complete description of intermediate vn algs} and Corollary \ref{cor: L(SL_2(Q)) is mHAP}.

\begin{proof}[Proof of Theorem \ref{thm: complete description of intermediate vn algs}]

\textbf{Step 1}: preparations and setting up notations.

Let $E$ be the trace preserving conditional expectation from $(L(\mathbb{Q}^2\rtimes G),\tau)$ onto $(P,\tau|_P)$, where $\tau$ denotes the canonical trace on $L(\mathbb{Q}^2\rtimes G)$. Note that $L(G)\subset P\subset A\rtimes G\subset L(\mathbb{Q}^2\rtimes G)$ and recall that $g.v$ denotes the matrix multiplication between $g\in G$ and $v\in \mathbb{Q}^2$.

Let $c_v:=u_v^*E(u_v)\in L(u_v^*Gu_v\cap G)'\cap L(\mathbb{Q}^2\rtimes G)$ for each $v\in\mathbb{Q}^2$. Besides, by a similar argument used to prove Claim 1 in Step 1 in the proof of Theorem \ref{thm: no intermediate subalgs}, we know $c_{e_1}\in L(\big(\begin{smallmatrix}
\mathbb{Q}\\0
\end{smallmatrix}\big)\rtimes \pm\big(\begin{smallmatrix}
1& \mathbb{Q}\\
0& 1
\end{smallmatrix}\big))$. Moreover, the following hold.
\begin{align}
\sigma_v(c_v)&=c_{-v}^*, ~\mbox{where $\sigma_v(\cdot):=u_v\cdot u_v^*$}.\label{eq: relation btw c_a and c_-a}\\
c_{g.v}&=\sigma_g(c_v)~\mbox{for all $g\in G$, where $\sigma_g(\cdot)=u_g\cdot u_g^*$}.\label{eq: relation btw c_g.a and c_a}\\
E(\sigma_v(c_v)u_{v+g.w})&=\sigma_v(c_v)u_{v+g.w}\sigma_g(c_w)~\mbox{for all $v$, $w\in \mathbb{Q}^2$ and $g\in G$}.\label{eq: key restriction for E(c_a) in proving maximality}
\end{align}
Indeed, the first identity can be proved by computing both sides of $E(u_v)=E(u_{-v})^*$.

For the second one, observe that $u_{g.v}c_{g.v}=E(u_{g.v})=E(u_gu_vu_g^{-1})=u_gE(u_v)u_g^*=\sigma_g(u_vc_v)=u_{g.v}\sigma_g(c_v)$.

For the last one, observe that $E(E(u_vu_g)u_w)=E(u_vu_g)E(u_w)$, then we compute both sides of this identity to deduce (\ref{eq: key restriction for E(c_a) in proving maximality}):
\begin{align*}
E(E(u_vu_g)u_w)&=E(E(u_v)u_gu_w)=E(u_vc_vu_gu_w)=E((u_vc_vu_v^*)u_v(u_gu_wu_g^*)u_g)\\
&=E(\sigma_v(c_v)u_{v+g.w})u_g,\\
E(u_vu_g)E(u_w)&=E(u_v)u_gE(u_w)=u_vc_vu_gu_wc_w=(u_vc_vu_v^*)u_v(u_gu_wu_g^*)(u_gc_wu_g^*)u_g\\
&=\sigma_v(c_v)u_{v+g.w}\sigma_g(c_w)u_g.
\end{align*}


Now, we may write 
\begin{align}\label{eq: def of c_e_1}
c_{e_1}=\sum_{x, y}\lambda_{x, y}\big(\begin{smallmatrix}
x\\
0
\end{smallmatrix}\big)\big(\begin{smallmatrix}
1&y\\
0&1
\end{smallmatrix}\big)+\sum_{x, y}\mu_{x, y}\big(\begin{smallmatrix}
x\\
0
\end{smallmatrix}\big)\big(\begin{smallmatrix}
-1&y\\
0&-1
\end{smallmatrix}\big).
\end{align}

\textbf{Step 2}: find restrictions on $\lambda_{x, y}$ and $\mu_{x, y}$ only using $LG\subset P\subset L(\Q^2\rtimes G)$.

First, note that $u_g\in P$ implies $\langle u_{e_1}-E(u_{e_1}), u_g \rangle=0$ for all $g\in G$.
For any $b\in \Q$, set $g=\big(\begin{smallmatrix}
1&b\\0&1
\end{smallmatrix}\big)$. Then, $0=\tau(u_{e_1}u_g^*)=\tau(E(u_{e_1})u_g^*)=\tau(u_{e_1}c_{e_1}u_g^*)=\lambda_{-1, b}$.
Similarly, set $g=\big(\begin{smallmatrix}
-1&b\\0&-1
\end{smallmatrix}\big)$, we deduce that $0=\mu_{-1, b}$.
Hence, we have shown the following hold:
\begin{align}\label{eq: lambda_-1, b=0=mu_-1, b}
\lambda_{-1, b}=0=\mu_{-1, b}~\mbox{for all}~ b\in\Q.
\end{align}

Second, plug $v=e_1$ into (\ref{eq: relation btw c_a and c_-a}), we get $\sigma_{e_1}(c_{e_1})=c_{-e_1}^*$. On the one hand, 
\begin{align*}
 \sigma_{e_1}(c_{e_1})=u_{e_1}c_{e_1}u_{e_1}^*\overset{(\ref{eq: def of c_e_1})}{=}
\sum_{x, y}\lambda_{x, y}\big(\begin{smallmatrix}
x\\
0
\end{smallmatrix}\big)\big(\begin{smallmatrix}
1&y\\
0&1
\end{smallmatrix}\big)+\sum_{x, y}\mu_{x, y}\big(\begin{smallmatrix}
x+2\\
0
\end{smallmatrix}\big)\big(\begin{smallmatrix}
-1&y\\
0&-1
\end{smallmatrix}\big).
\end{align*}
On the other hand,
\begin{align*}
c_{-e_1}^*
\overset{(\ref{eq: relation btw c_g.a and c_a})}{=}(u_{-id}c_{e_1}u_{-id}^*)^*
\overset{(\ref{eq: def of c_e_1})}{=}\sum_{x, y}\overline{\lambda_{x, y}}\big(\begin{smallmatrix}
x\\
0
\end{smallmatrix}\big)\big(\begin{smallmatrix}
1&-y\\
0&1
\end{smallmatrix}\big)+\sum_{x, y}\overline{\mu_{x, y}}\big(\begin{smallmatrix}
-x\\
0
\end{smallmatrix}\big)\big(\begin{smallmatrix}
-1&-y\\
0&-1
\end{smallmatrix}\big).
\end{align*}
By comparing the above two expressions, we deduce that
\begin{align}\label{eq: conjugation restriction on coefficients}
\lambda_{x, y}=\overline{\lambda_{x, -y}}~\mbox{and}~
\mu_{x, y}=\overline{\mu_{-x-2, -y}}~\mbox{for all $x$, $y\in \Q$}.
\end{align}

Third, by plugging $g=id$, $v=e_1$ and $w=e_2$ into (\ref{eq: key restriction for E(c_a) in proving maximality}), we compute both sides of (\ref{eq: key restriction for E(c_a) in proving maximality}) to get the following
\begin{align*}
\mbox{RHS of (\ref{eq: key restriction for E(c_a) in proving maximality})}
&=\sigma_{e_1}(c_{e_1})u_{e_1+e_2}c_{e_2}\\
&=u_{e_1}c_{e_1}u_{e_2}\big(\begin{smallmatrix}
0&-1\\1&0
\end{smallmatrix}\big)c_{e_1}\big(\begin{smallmatrix}
0&1\\-1&0
\end{smallmatrix}\big)~\mbox{(by def. of $\sigma_{e_1}(\cdot)$, (\ref{eq: relation btw c_g.a and c_a}) and $e_2=\big(\begin{smallmatrix}0&-1\\1&0 \end{smallmatrix}\big)e_1$)}\\   
&\overset{(\ref{eq: def of c_e_1})}{=}\sum_{x, y, a, b}\lambda_{x, y}\lambda_{a, b}\big(\begin{smallmatrix}
x+1+(a+1)y\\ 1+a
\end{smallmatrix}\big)\big(\begin{smallmatrix}
1-by&y\\-b&1
\end{smallmatrix}\big)
+\sum_{x, y, a, b}\lambda_{x, y}\mu_{a, b}\big(\begin{smallmatrix}
x+1+(a+1)y\\ 1+a
\end{smallmatrix}\big)\big(\begin{smallmatrix}
-1-by&-y\\-b&-1
\end{smallmatrix}\big)\\
&+\sum_{x, y, a, b}\mu_{x, y}\lambda_{a, b}\big(\begin{smallmatrix}
x+1+(a+1)y\\ -1-a
\end{smallmatrix}\big)\big(\begin{smallmatrix}
-1-by&y\\b&-1
\end{smallmatrix}\big)+\sum_{x, y, a, b}\mu_{x, y}\mu_{a, b}\big(\begin{smallmatrix}
x+1+(a+1)y\\ -1-a
\end{smallmatrix}\big)\big(\begin{smallmatrix}
1-by&-y\\b&1
\end{smallmatrix}\big).
\end{align*}
\begin{align*}
\mbox{LHS of (\ref{eq: key restriction for E(c_a) in proving maximality})}=E(u_{e_1}c_{e_1}u_{e_2})
\overset{(\ref{eq: def of c_e_1})}{=}\sum_{x, y}\lambda_{x, y}E\big[\big(\begin{smallmatrix}
1+x+y\\1
\end{smallmatrix}\big)\big]\big(\begin{smallmatrix}
1&y\\0&1
\end{smallmatrix}\big)+\sum_{x, y}\mu_{x, y}E\big[\big(\begin{smallmatrix}
1+x+y\\-1
\end{smallmatrix}\big)\big]\big(\begin{smallmatrix}
-1&y\\0&-1
\end{smallmatrix}\big).
\end{align*}

Then, notice that we can compute both $E\big[\big(\begin{smallmatrix}
1+x+y\\1
\end{smallmatrix}\big)\big]$ and $E\big[\big(\begin{smallmatrix}
1+x+y\\-1
\end{smallmatrix}\big)\big]$ in terms of $c_{e_1}$ by first applying the definition $E(u_{g.v})=u_{g.v}c_{g.v}$ and then applying (\ref{eq: relation btw c_g.a and c_a}) by taking $v=e_1$ and $g=\big(\begin{smallmatrix}
1+x+y&x+y\\1&1
\end{smallmatrix}\big)$ and $g=\big(\begin{smallmatrix}
1+x+y&-x-y\\-1&1
\end{smallmatrix}\big)$ respectively.
After these substitutions, we deduce that
\begin{align*}
\mbox{LHS of (\ref{eq: key restriction for E(c_a) in proving maximality})}
&\overset{(\ref{eq: def of c_e_1})}{=}\sum_{x, y, a, b}\lambda_{x, y}\lambda_{a, b}\big(\begin{smallmatrix}
(a+1)(1+x+y)\\a+1
\end{smallmatrix}\big)\big(\begin{smallmatrix}
1-b(1+x+y)& y+b(x+1)(1+x+y)\\
-b&1+b(x+1)
\end{smallmatrix}\big)\\
&+\sum_{x, y, a, b}\lambda_{x, y}\mu_{a, b}\big(\begin{smallmatrix}
(a+1)(1+x+y)\\a+1
\end{smallmatrix}\big)\big(\begin{smallmatrix}
-1-b(1+x+y)& -y+b(x+1)(1+x+y)\\
-b&-1+b(x+1)
\end{smallmatrix}\big)\\
&+\sum_{x, y, a, b}\mu_{x, y}\lambda_{a, b}\big(\begin{smallmatrix}
(a+1)(1+x+y)\\-a-1
\end{smallmatrix}\big)\big(\begin{smallmatrix}
-1-b(1+x+y)& y-b(x+1)(1+x+y)\\
b&-1+b(x+1)
\end{smallmatrix}\big)\\
&+\sum_{x, y, a, b}\mu_{x, y}\mu_{a, b}\big(\begin{smallmatrix}
(a+1)(1+x+y)\\-a-1
\end{smallmatrix}\big)\big(\begin{smallmatrix}
1-b(1+x+y)& -y-b(x+1)(1+x+y)\\
b&1+b(x+1)
\end{smallmatrix}\big).
\end{align*}

Now, for any given $s$, $a\neq -1$, $b$ and $y$ in $\Q$, we compare the coefficients of both sides of (\ref{eq: key restriction for E(c_a) in proving maximality}) in front of the term $\big(\begin{smallmatrix}
s\\a+1
\end{smallmatrix}\big)\big(\begin{smallmatrix}
1-by&y\\-b&1
\end{smallmatrix}\big)$.

For the RHS, the coefficient is $\lambda_{s-1-(a+1)y, y}\lambda_{a, b}+\mu_{s-1-(a+1)y, -y}\mu_{-2-a, -b}$, which equals the coefficient on the LHS, i.e.
\begin{align}\label{system of equations: eq3-eq4}
\begin{cases}
\lambda_{\frac{s}{a+1}-y-1, y}\lambda_{a, 0}+\mu_{-a-2, 0}\mu_{y-1-\frac{s}{a+1}, -y}~&~\mbox{if}~ s\neq 0~\mbox{and}~ b=0 \\
0, & ~\mbox{if}~y\neq  \frac{s}{a+1}+\frac{2}{b}, \frac{s}{a+1}~\mbox{and}~ bs\neq 0.
\end{cases}
\end{align}

Next,  for any given $s$, $a\neq -1$, $b$ and $y$ in $\Q$, we compare the coefficients of both sides of (\ref{eq: key restriction for E(c_a) in proving maximality}) in front of the term $\big(\begin{smallmatrix}
s\\-a-1
\end{smallmatrix}\big)\big(\begin{smallmatrix}
-1-by&y\\b&-1
\end{smallmatrix}\big)$.

For the RHS, the coefficient is $\lambda_{s-1-(a+1)y, -y}\mu_{-2-a, -b}+\mu_{s-1-(a+1)y, y}\lambda_{a, b}$, which equals the coefficient on the LHS, i.e.
\begin{align}\label{system of equations: eq7-eq8}
\begin{cases}
\lambda_{y-\frac{s}{a+1}-1, -y}\mu_{-a-2, 0} +\mu_{\frac{s}{a+1}-1-y, y}\lambda_{a, 0}, &~\mbox{if}~ s\neq 0~\mbox{and}~ b=0\\
 0, & ~\mbox{if}~y\neq  \frac{s}{a+1}-\frac{2}{b}, \frac{s}{a+1}~\mbox{and}~ bs\neq 0.
\end{cases}
\end{align}

Till now, we have only used the assumption that $LG\subset P\subset L(\Q^2\rtimes G)$. 
In order to solve the above system of equations (\ref{system of equations: eq3-eq4})-(\ref{system of equations: eq7-eq8}) effectively, we explore the fact that $LG\subset P\subset A\rtimes G$ to get direct relations between $\lambda$-coefficients and $\mu$-coefficients.

\textbf{Step 3}: find more restrictions on $\lambda_{x, y}$ and $\mu_{x, y}$ by exploring $LG\subset P\subset A\rtimes G$.

First, by combining the definition of $A$, (\ref{eq: def of c_e_1}) and the fact $E(u_{e_1})=u_{e_1}c_{e_1}\in P$, we get 
\begin{align}\label{eq: lambda is symmetric along x=-1}
\lambda_{x, y}=\lambda_{-2-x, y}~\mbox{and}~\mu_{x, y}=\mu_{-2-x, y},\forall~x,~\forall~y.
\end{align}

Then, notice that by a similar argument used in the proof of Corollary \ref{cor: max Haagerup for LG}, we know that $(qLG\subset q(A\rtimes G))\cong (L(\bar{G})\subset A\rtimes \bar{G})$, where $q=\frac{u_{id}+u_{-id}}{2}$. 
Then by Theorem \ref{thm: no intermediate subalgs}, we know that $qLG$ is maximal inside $q(A\rtimes G)$, hence $qP\in\{qLG, q(A\rtimes G)\}$. Our goal is to show that $(1-q)P=(1-q)LG$ or $(1-q)(A\rtimes G)$.

\textbf{Claim}:  $\lambda_{x, y}+\mu_{x, -y}=0$, $\forall~x$, $\forall~y\neq 0$ and $\lambda_{x,0}+\mu_{x, 0}=0$ if $x\neq 0$ or $-2$.
\begin{proof}[Proof of the Claim]
We need to consider two cases.

\textbf{Case 1}: $qP=qLG$.

By the definition of $q$, one can check that 
\[qLG\subseteq \{\sum_g\lambda_gu_g: \lambda_g=\lambda_{-g}\in \C,~\forall~g\in G\}\cap \ell^2(G).\]
Then, by using (\ref{eq: def of c_e_1}), we get $qE(u_{e_1})$, as an element in $qP=qLG$, is equal to
\begin{align*}
\quad qu_{e_1}c_{e_1}\overset{(\ref{eq: def of c_e_1})}{=}\sum_{x, y}\frac{\lambda_{x-1,y}+\mu_{-x-1,-y}}{2}\big(\begin{smallmatrix}
x\\0
\end{smallmatrix}\big)\big(\begin{smallmatrix}
1&y\\0&1
\end{smallmatrix}\big)+\sum_{x, y}\frac{\lambda_{-x-1, -y}+\mu_{x-1, y}}{2}\big(\begin{smallmatrix}
x\\0
\end{smallmatrix}\big)\big(\begin{smallmatrix}
-1&y\\0&-1
\end{smallmatrix}\big).
\end{align*}
Therefore, $0=\lambda_{x-1, y}+\mu_{-x-1, -y}\overset{(\ref{eq: lambda is symmetric along x=-1})}{=}\lambda_{x-1, y}+\mu_{x-1, -y}$ for all $x\neq 0$; equivalently, $\lambda_{x, y}+\mu_{x, -y}=0$ for all $x\neq -1$ and all $y$. By (\ref{eq: lambda_-1, b=0=mu_-1, b}), this finishes the proof.

\textbf{Case 2}: $qP=q(A\rtimes G)$.

Note that $\langle u_{e_1}-E(u_{e_1}), P \rangle=0$, in particular, we have
\begin{align*}
\langle u_{e_1}-E(u_{e_1}), q(\big(\begin{smallmatrix}
a\\0
\end{smallmatrix}\big)+\big(\begin{smallmatrix}
-a\\0
\end{smallmatrix}\big))u_g \rangle=0,~\forall~0\neq a\in\Q,~\forall~g\in G.
\end{align*}
For any $s\in \mathbb{Q}$, set $g=\big(\begin{smallmatrix}
1&s\\0&1
\end{smallmatrix}\big)$, then a calculation by using the above identity and (\ref{eq: def of c_e_1}) shows that
$\lambda_{a-1, s}+\lambda_{-a-1, s}+\mu_{a-1, -s}+\mu_{-a-1, -s}=\begin{cases}
1, &~\mbox{if}~ s=0~\mbox{and}~a=\pm 1\\
0, &~\mbox{otherwise}.
\end{cases}$

By using (\ref{eq: lambda is symmetric along x=-1}), we know that the above is equivalent to 
\begin{align*}
\lambda_{a-1, s}+\mu_{a-1, -s}&=0,~ \forall ~a, \forall~ s\neq 0,\\
\lambda_{x, 0}+\mu_{x, 0}&=0,~ \forall~ x\neq 0, -2,\\
\lambda_{x, 0}+\mu_{x, 0}&=\frac{1}{2}, ~\mbox{if}~x=0, -2.
\end{align*}
Therefore, the Claim is proved and we always have 
\begin{align}\label{eq: relation between lambda and mu by symmetry}
\lambda_{x,y}+\mu_{x,-y}=0,~\forall~ x~ \forall~ y\neq 0~\mbox{and}~\lambda_{x, 0}+\mu_{x, 0}=0,~\forall~ x\neq 0, -2.
\end{align}
\end{proof}
\textbf{Step 4}: solve for the $\lambda$- and $\mu$-coefficients.

By (\ref{eq: relation between lambda and mu by symmetry})-(\ref{eq: lambda is symmetric along x=-1}), we can simplify the 2nd case in (\ref{system of equations: eq3-eq4})-(\ref{system of equations: eq7-eq8}) to the following respectively:
\begin{align*}
\lambda_{s-1-(a+1)y, y}\lambda_{a, b}=0,~ \forall~ bs\neq 0, ~\forall~ a\neq -1, ~\forall~ y\neq 0, \frac{s}{a+1}, \frac{s}{a+1}+\frac{2}{b},\\
\lambda_{s-1-(a+1)y, -y}\lambda_{a, b}=0,~ \forall~ bs\neq 0,~ \forall ~a\neq -1, ~\forall~ y\neq 0, \frac{s}{a+1}, \frac{s}{a+1}-\frac{2}{b}.
 \end{align*}
Assume $\lambda_{a, b}\neq 0$ for some $a\neq -1$ and $b\neq 0$. Clearly, $(s-1-(a+1)y, y)=(a, b)$ iff $s=(a+1)(1+y), y=b$; similarly, $(s-1-(a+1)y, -y)=(a, b)$ iff $s=(a+1)(1+y), y=-b$. One can check that for such choices of $(s, y)$, the above expressions become
\begin{align*}
\lambda_{a, b}=0~\mbox{if}~b\neq 0, -1, -2~\mbox{or}~b\neq 0, 1, 2.
\end{align*}
Therefore, we deduce that $\lambda_{a, b}=0$ for all $b\neq 0$, a contradiction. Hence, we have proved that
\begin{align}\label{eq: lambda_a, b=0 for a neq -1 and b nonzero}
\lambda_{a, b}=0,~\forall~a\neq -1~\mbox{and}~b\neq 0.
\end{align}

From (\ref{eq: lambda_-1, b=0=mu_-1, b}) and (\ref{eq: lambda_a, b=0 for a neq -1 and b nonzero}), we deduce $\lambda_{a, b}=0$ unless $b=0$. We are left to solve for $\lambda_{a, 0}$.

Plugging $y=0$ into the 1st case in both (\ref{system of equations: eq3-eq4}) and (\ref{system of equations: eq7-eq8}), we deduce that 
\begin{align*}
\lambda_{s-1, 0}\lambda_{a,0}+\mu_{s-1, 0}\mu_{-a-2, 0}=\lambda_{\frac{s}{a+1}-1,0}\lambda_{a, 0}+\mu_{-a-2, 0}\mu_{-1-\frac{s}{a+1}, 0},~\forall~s\neq 0, a\neq -1,\\
\lambda_{s-1, 0}\mu_{-2-a, 0}+\mu_{s-1,0}\lambda_{a, 0}=\lambda_{-\frac{s}{a+1}-1, 0}\mu_{-a-2, 0}+\mu_{\frac{s}{a+1}-1, 0}\lambda_{a, 0},~\forall~s\neq 0, a\neq -1.
\end{align*}

Again, by using (\ref{eq: relation between lambda and mu by symmetry})-(\ref{eq: lambda is symmetric along x=-1}), we can simplify the above expressions to the following:
\begin{align*}
(\lambda_{s-1, 0}-\lambda_{\frac{s}{a+1}-1, 0})\lambda_{a, 0}=(\mu_{s-1,0}-\mu_{\frac{s}{a+1}-1,0})\lambda_{a, 0},~\forall s\neq0,~\forall~ a\neq -1, 0, -2.
\end{align*}
Assume $\lambda_{a, 0}\neq 0$ for some $a\not\in \{-2, -1, 0\}$, then the above identity becomes
$\lambda_{s-1, 0}-\mu_{s-1, 0}=\lambda_{\frac{s}{a+1}-1, 0}-\mu_{\frac{s}{a+1}-1, 0}$ for all $s\neq 0$; equivalently,
$\lambda_{s-1, 0}-\mu_{s-1, 0}=\lambda_{s(a+1)-1, 0}-\mu_{s(a+1)-1, 0}$ for all $s\neq 0$.
Hence,
\begin{align*}
\lambda_{s-1, 0}-\mu_{s-1, 0}&=\lambda_{s(a+1)-1, 0}-\mu_{s(a+1)-1, 0}\\
&=\lambda_{s(a+1)^2-1, 0}-\mu_{s(a+1)^2-1, 0}\\
&=\cdots\\
&=\lambda_{s(a+1)^n-1, 0}-\mu_{s(a+1)^n-1, 0},~ \forall~n\geq 1.
\end{align*}
Since $|a+1|\neq 1$, $s\neq 0$ and $a\in\Q$, we know $s(a+1)^n-1\neq s(a+1)^m-1$ for all $n\neq m$. Moreover, as $\sum_s|\lambda_{s,0}-\mu_{s,0}|^2\leq 2(\sum_s|\lambda_{s,0}|^2+\sum_s|\mu_{s,0}|^2)<\infty$, the above expression implies $\lambda_{s-1, 0}-\mu_{s-1,0}=0$ for all $s\neq 0$. Now by taking $s=a+1$, we get $0=\lambda_{a, 0}-\mu_{a, 0}=2\lambda_{a, 0}$ by (\ref{eq: relation between lambda and mu by symmetry}), a contradiction. Hence, $\lambda_{a, 0}=0$ for all $a\not\in \{0, -1, -2\}$.

To sum up, (\ref{eq: lambda_-1, b=0=mu_-1, b}) and the above tell us that $\lambda_{x, y}=0$ if $(x, y)\neq (0, 0)$, $(-2, 0)$. By (\ref{eq: relation between lambda and mu by symmetry}), this also implies that $\mu_{x, y}=0$ if $(x, y)\neq (0, 0)$, $(-2, 0)$. 
Also note that $\lambda_{0, 0}=\lambda_{-2, 0}:=\lambda$, $\mu_{0, 0}=\mu_{-2, 0}:=\mu$ by (\ref{eq: lambda is symmetric along x=-1}). Moreover, both $\lambda$ and $\mu$ are real numbers by (\ref{eq: conjugation restriction on coefficients}). Besides, $\lambda=-\mu$ if $qP=qLG$ and $\lambda+\mu=\frac{1}{2}$ if $qP=q(A\rtimes G)$ from the proof of (\ref{eq: relation between lambda and mu by symmetry}).

\textbf{Step 5}: prove $(1-q)P=(1-q)LG$ or $(1-q)(A\rtimes G)$.

From Step 4, we can simplify (\ref{eq: def of c_e_1}) to get
\begin{align*}
E(u_{e_1})=u_{e_1}c_{e_1}=\lambda(u_{e_1}+u_{-e_1})+\mu(u_{e_1}+u_{-e_1})u_{-id}.
\end{align*}
From the above and (\ref{eq: relation btw c_g.a and c_a}), we infer that 
\begin{align*}
E(u_v)=\lambda(u_{v}+u_{-v})+\mu(u_v+u_{-v})u_{-id},~\forall~0\neq v\in \Q^2.
\end{align*}
So $(1-q)P\ni (1-q)E(u_v+u_{-v})=(2\lambda-2\mu)(u_v+u_{-v})(1-q)$. Hence $(1-q)(u_v+u_{-v})\in (1-q)P$ if $\lambda\neq \mu$. Since $\{(1-q)(u_v+u_{-v}): v\in \Q^2\}$ linearly spans a dense subset in $(1-q)A$ and $q\in P$, this implies that $(1-q)P=(1-q)(A\rtimes G)$ if $\lambda\neq \mu$. If $\lambda=\mu$, then it is easy to see $(1-q)P=(1-q)LG$. So we have proved that $(1-q)P\in \{(1-q)LG, (1-q)(A\rtimes G)\}$.
\end{proof}

Now, let us prove Corollary \ref{cor: L(SL_2(Q)) is mHAP}.

\begin{proof}[Proof of Corollary \ref{cor: L(SL_2(Q)) is mHAP}]
Recall that $G=SL_2(\mathbb{Q})$. Let $P$ be any intermediate von Neumann subalgebra with Haagerup property between $LG$ and $L(\Q^2\rtimes G)$. Define $\phi\in Aut(L(\Q^2\rtimes G))$ by setting $\phi=Ad(u_g)$, where $g=\big(\begin{smallmatrix}
-1&0\\
0&-1
\end{smallmatrix}\big)$. Clearly, $\phi^2=id$ and $Fix(\phi)=A\rtimes G$. Let $P_0=P\cap Fix(\phi)\subset A\rtimes G$. Moreover, notice that $\phi(P)=P$ as $u_g\in P$, so we can view $\phi$ as an automorphism on $P$.

Since $P$ has Haagerup property, we know $P_0$ also has Haagerup property. Then, using Theorem \ref{thm: complete description of intermediate vn algs}, we know that $P_0=LG$. Indeed, this is because both $qP_0$ and $(1-q)P_0$ have Haagerup property by \cite[Theorem 2.3]{jolissaint}, which implies that $qP_0=q(LG)$ and $(1-q)P_0=(1-q)(LG)$. To see these two equalities hold, notice that $A\rtimes G$ contains $(A\cap L(\Z^2))\rtimes SL_2(\Z)$, which has relative (T) with respect to the diffuse subalgebra $A\cap L(\Z^2)$ by the proof of Corollary \ref{cor: max Haagerup for LG}. Hence $(1-q)((A\cap L(\Z^2))\rtimes SL_2(\Z))$ (resp. $q((A\cap L(\Z^2))\rtimes SL_2(\Z))$) has relative (T) with respect to $(1-q)(A\cap L(\Z^2))$ (resp. $q(A\cap L(\Z^2))$), say by \cite[Proposition 4.7]{popa_annals}. This implies neither $(1-q)((A\cap L(\Z^2))\rtimes SL_2(\Z))$ nor $q((A\cap L(\Z^2))\rtimes SL_2(\Z))$ has the Haagerup property, so neither $(1-q)(A\rtimes G)$ nor $q(A\rtimes G)$ has the Haagerup property.

Now, we are left to show that $P=LG$. The strategy is similar to the proof of \cite[Corollary 3.14]{chifan_das}.

Since $P'\cap P\subset (LG)'\cap L(\mathbb{Q}^2\rtimes G)=\mathbb{C}+\mathbb{C}u_{-id}= q\mathbb{C}\oplus (1-q)\mathbb{C}$, 
the center of $P$ has dimension less or equal to two. Note that $P_0=P^{\{id, \phi\}}$, 
we deduce the Pimsner-Popa index $[P: P_0]<\infty$ by \cite[Theorem 3.2]{jones_xu}. 

To see the above holds, we first take this opportunity to correct several misprints in \cite[Theorem 3.2]{jones_xu}: (1) in the statement of the theorem, $M$ should be $N$, i.e. $\mathcal{A}$ is a finite set of automorphisms of the finite factor $N$; (2) in its proof, to make $P\cap Q=\{\oplus_{\alpha}x|x\in N^{\mathcal{A}}\}$ hold, one implicitly assumes $id\in \mathcal{A}$; (3) in the 4th line of the proof, the 2nd $N$ should be $M$, i.e. it should read as ``\ldots iff the index of $P\cap Q$ in $M$ is finite.". After correcting these misprints, we notice that the proof of \cite[Theorem 3.2]{jones_xu} still works under the weaker assumption that $N$ (in this theorem) is a direct sum of finitely many finite factors as it relies on \cite[Theorem 3.1]{jones_xu}. Finally, we can apply this theorem by taking $N=P$ and $\mathcal{A}=\{id, \phi\}$. Indeed, this is because $\phi\in Aut(P)$ and the spectrum of the operator $\phi$ is finite. 
 
Since $P_0=LG\subset P\subset L(\Q^2\rtimes G)$, the above implies $P\subset \mathcal{QN}_{L(\Q^2\rtimes G)}(LG)''$. Here, for any finite von Neumann algebras $N\subset M$, $\mathcal{QN}_M(N)$ denotes the quasi-normalizers which is defined as the *-subalgebra of $M$ consisting of all elements $x\in M$ such that there exist $x_1$, $x_2, \dots, x_k\in M$ such that $Nx\subset \sum_ix_iN$ and $xN\subset \sum_iNx_i$ \cite{Po99}.

By Proposition 6.10 in the arXiv version of \cite{ioana_duke}, we know that $\mathcal{QN}_{L(\Q^2\rtimes G)}(LG)''=L^{\infty}(X_c)\rtimes G$, where $G\curvearrowright X_c$ is the maximal compact factor of $G\curvearrowright \widehat{\Q^2}$. Since $G\curvearrowright \widehat{\Q^2}$ is weakly mixing by the proof of Corollary \ref{cor: weakly mixing prime action}, we deduce that $X_c$ is a singleton, say by the proof of \cite[Theorem 2.28]{kerrli_book}, hence $LG\subset P\subset LG$, i.e. $P=LG$.   
This proves that $LG$ is maximal Haagerup inside $L(\Q^2\rtimes G)$.
\end{proof}

\section{Complete description of intermediate von Neumann subalgebras: $\Z$-coefficient}\label{section: Z-case}

In this section, we check that after certain modifications, results in the previous section for $\Q$-coefficient groups also hold for the corresponding $\Z$-coefficient groups. To state the results precisely, we need the following notation.
Denote by $B$ the following von Neumann subalgebra of $L(\mathbb{Z}^2)$:
\[B=\big\{\sum_{x, y}\lambda_{x, y}u_{x, y}: \lambda_{x, y}=\lambda_{-x,-y},~\forall~ x, y\in \Z\big\}\cap L(\mathbb{Z}^2).\]
Observe that $B\rtimes SL_2(\Z)\cong L^{\infty}(\widehat{\Z^2}/{\sim})\rtimes SL_2(\Z)$ and $B\rtimes PSL_2(\Z)\cong L^{\infty}(\widehat{\Z^2}/{\sim})\rtimes PSL_2(\Z)$. Here, $SL_2(\Z)\curvearrowright \widehat{\Z^2}/{\sim}$ is the quotient action of $SL_2(\Z)\curvearrowright \widehat{\Z^2}$ by modding out the relation $\phi\sim\phi'$, where $\phi$, $\phi'\in \widehat{\Z^2}\cong \mathbb{T}^2\cong [-\frac{1}{2}, \frac{1}{2}]^2$ and $\phi'(x, y):=\phi(-x, -y)$ for all $(x, y)\in \mathbb{Z}^2$. Then notice that $SL_2(\Z)\curvearrowright \widehat{\Z^2}/{\sim}$ descends to a $PSL_2(\Z)$-action, i.e. $PSL_2(\Z)\curvearrowright \widehat{\Z^2}/{\sim}$ by modding out the kernel of the action.

The main result in this section is the following theorem.
\begin{theorem}\label{thm: SL_2(Z)-thm}
If $P$ is a von Neumann algebra between $L(SL_2(\Z))$ and $B\rtimes SL_2(\Z)$, then \[P=q[(B\cap L(n\Z^2))\rtimes SL_2(\mathbb{Z})]\oplus (1-q)[(B\cap L(m\Z^2))\rtimes SL_2(\mathbb{Z})]~\mbox{for two integers $n$, $m$},\] where $q=\frac{u_{id}+u_{-id}}{2}$ and  $id$ denotes the identity matrix in $SL_2(\Z)$.
\end{theorem}

Similar to the $\mathbb{Q}$-coefficient case, we need to first deal with the $PSL_2(\mathbb{Z})$-action.

\begin{theorem}\label{thm: PSL_2(Z)-thm}
If $P$ is a von Neumann algebra between $L(PSL_2(\Z))$ and $B\rtimes PSL_2(\Z)$, then $P=(B\cap L(n\Z^2))\rtimes PSL_2(\Z)$ for an integer $n$. 
\end{theorem}

As an application of Theorem \ref{thm: SL_2(Z)-thm}, we get the following corollary.
\begin{corollary}\label{cor: L(SL_2(Z)) is mHAP}
$L(SL_2(\Z))$ is a maximal Haagerup von Neumann subalgebra in $L(\Z^2\rtimes SL_2(\Z))$.
\end{corollary}
\begin{proof}
Note that if $m\neq 0$, then $B\cap L(m\Z^2)\subset (B\cap L(m\Z^2))\rtimes SL_2(\mathbb{Z})$ has relative (T) by \cite[Lemma 3.5]{js}; equivalently, both $q[B\cap L(m\Z^2)]\subset q[(B\cap L(m\Z^2))\rtimes SL_2(\mathbb{Z})]$ and $(1-q)[B\cap L(m\Z^2)]\subset (1-q)[(B\cap L(m\Z^2))\rtimes SL_2(\mathbb{Z})]$ have relative (T) by \cite[Proposition 4.7]{popa_annals}.
The rest proof is almost identical to the proof of Corollary \ref{cor: L(SL_2(Q)) is mHAP}. 
We leave it as an exercise.
\end{proof}
\paragraph*{\textbf{On the difference between $\Z$ and $\Q$-coefficient}}
The proof of Theorem \ref{thm: PSL_2(Z)-thm} (resp. Theorem \ref{thm: SL_2(Z)-thm}) follows the proof of Theorem \ref{thm: no intermediate subalgs} (resp. Theorem \ref{thm: complete description of intermediate vn algs}) closely. In fact, most parts of the proof for $\Q$-coefficient case still work verbatim for the $\Z$-coefficient case. The key difference lies in the fact that the affine action ${\mathbf{k}}^2\rtimes SL_2(\mathbf{k})\curvearrowright \mathbf{k}^2$ is 2-transitive for $\mathbf{k}=\Q$ but not for $\mathbf{k}=\Z$. Due to the failure of 2-transitivity for $\mathbf{k}=\Z$ case, $E(u_{ne_1}+u_{-ne_1})$ (resp. $E(u_{ne_1})$) is not determined directly by $E(u_{e_1}+u_{-e_1})$ (resp. $E(u_{e_1})$) for all $n\geq 2$, where $\pm ne_1=\big(\begin{smallmatrix}
\pm n\\0
\end{smallmatrix}\big)$ and $E$ denotes the trace preserving conditional expectation onto the mysterious intermediate von Neumann subalgebra of $B\rtimes PSL_2(\Z)$ (resp. $B\rtimes SL_2(\Z)$). Instead, we do the computation, which was used for determining $E(u_{e_1}+u_{-e_1})$ (resp. $E(u_{e_1})$), for each $n$ inductively. Then, this computation is combined with Packer's result \cite{packer} (see also \cite{suzuki}) and the known result on complete description of all factors of the action $SL_2(\Z)\curvearrowright \widehat{\Z^2}\cong \mathbb{T}^2$ \cite[Lemma 3.5]{js} (see also \cite[Example 5.9]{witte}, \cite[Theorem 2.3]{park}) to finish the proof.

\begin{proof}[Proof of Theorem \ref{thm: PSL_2(Z)-thm}]
We first make the following claim:

\textbf{Claim 1:} For each $n\geq 1$, $E(u_{ne_1}+u_{-ne_1})=\lambda_n(u_{ne_1}+u_{-ne_1})$ for some scalar $\lambda_n$.
\begin{proof}[Proof of Claim 1]
To prove this claim, we first observe that for $n=1$, the proof is almost identical to the proof of Theorem \ref{thm: no intermediate subalgs}.

Indeed, the main argument in Step 1 there still works verbatim if we replace $\Q$ by $\Z$, so one can still write $E(u_{e_1}+u_{-e_1})=\sum_{x, y\in\Q}\lambda_{x, y}\big(\begin{smallmatrix}
x\\0
\end{smallmatrix}\big)\big(\begin{smallmatrix}
1&y\\0&1
\end{smallmatrix}\big)$, where $\lambda_{x, y}=\lambda_{-x, y}$ for all $(x, y)\in\Q^2$ and we may assume $\lambda_{x,y}=0$ if $x$ or $y\in\Q\setminus\Z$. Then, one can check the proofs in Step 2 and Step 3 still work to show that $\lambda_{x, y}=0$ unless $(x, y)=(\pm 1, 0)$ and hence we deduce $E(u_{e_1}+u_{e_{-1}})=\lambda(u_{e_1}+u_{-e_1})$ for some scalar $\lambda$ (and in fact $\lambda=0$ or $1$).

Next, assume the claim holds for all $n<k$, and let us check the claim for $n=k$.

Denote by 
\begin{align*}
I_1&:=\{1\leq i<k: E(u_{ie_1}+u_{-ie_1})=0\},\\
I_2&:=\{1\leq i<k: E(u_{ie_1}+u_{-ie_1})=\lambda_i(u_{ie_1}+u_{-ie_1})~\mbox{for some}~\lambda_i\neq 0\}.
\end{align*}
Without loss of generality, we may assume that $I_2=\emptyset$.
Indeed, assume not, then there exists some $i<k$ such that $u_{ie_1}+u_{-ie_1}=E(u_{ie_1}+u_{-ie_1})/\lambda_i\in P$.
Clearly, this implies that $(B\cap L(i\Z^2))\rtimes PSL_2(\Z)\subset P\subset B\rtimes PSL_2(\Z)\cong L^{\infty}(\widehat{\Z^2}/{\sim})\rtimes PSL_2(\Z)$. As $(B\cap L(i\Z^2))\rtimes PSL_2(\Z)\cong L^{\infty}(\widehat{i\Z^2}/{\sim})\rtimes PSL_2(\Z)$, where $PSL_2(\Z)\curvearrowright\widehat{i\Z^2}/{\sim}$ denotes the factor of $PSL_2(\Z)\curvearrowright \widehat{\Z^2}/{\sim}$ induced by the $SL_2(\Z)$-module inclusion $i\Z^2\hookrightarrow \Z^2$.
Since both $PSL_2(\Z)\curvearrowright \widehat{\Z^2}/{\sim}$ and $PSL_2(\Z)\curvearrowright \widehat{i\Z^2}/{\sim}$ are free actions, we deduce that $P=L^{\infty}(Z)\rtimes PSL_2(\Z)$ for some intermediate factor $PSL_2(\Z)\curvearrowright \widehat{\Z^2}/{\sim}\twoheadrightarrow Z\twoheadrightarrow \widehat{i\Z^2}/{\sim}$ by \cite{packer, suzuki}. We may replace the acting group by $SL_2(\Z)$ and apply \cite[Lemma 3.5]{js} to deduce that $(PSL_2(\Z)\curvearrowright \widehat{\Z^2}/{\sim}\twoheadrightarrow Z\twoheadrightarrow \widehat{i\Z^2}/{\sim})\cong (PSL_2(\Z)\curvearrowright \widehat{\Z^2}/{\sim}\twoheadrightarrow \widehat{m\Z^2}/{\sim}\twoheadrightarrow \widehat{i\Z^2}/{\sim})$ for some integer $m$ with $m\mid i$. In other words, $P=(B\cap L(m\Z^2))\rtimes PSL_2(\Z)$. Clearly, this implies that $E(u_{ke_1}+u_{-ke_1})=\lambda_k(u_{ke_1}+u_{-ke_1})$ for some scalar $\lambda_k$ and the induction step is finished.

From now on, we assume $I_2=\emptyset$ and prove Claim 1 for $n=k$.

Clearly, Step 1 in the proof of Theorem \ref{thm: no intermediate subalgs} still works to show that $E(u_{ke_1}+u_{-ke_1})=\sum_{x, y}\lambda_{x, y}\big(\begin{smallmatrix} 
x\\0
\end{smallmatrix}\big)\big(\begin{smallmatrix}
1&y\\0&1\\
\end{smallmatrix}\big)$ for some scalars $\lambda_{x, y}$ satisfying $\lambda_{x, y}=\lambda_{-x, y}$ for all $x, y$.

Now, we repeat the calculation in Step 2 in the proof of Theorem \ref{thm: no intermediate subalgs} but for $E(u_{ke_1}+u_{-ke_1})$. We sketch the calculation below with focus on the modification needed.

Note that $\lambda_{0, b}=0$ for all $b\in \mathbb{Z}$ by a similar argument to deduce (\ref{eq: 2nd restriction}).

We will compute both sides of the identity 
\begin{align}\label{eq: Z-case: E(E() )=E()E()}
E(E(u_{ke_1}+u_{-ke_1})(u_{ke_2}+u_{-ke_2}))=E(u_{ke_1}+u_{-ke_1})E(u_{ke_2}+u_{-ke_2}),
\end{align}
where $\pm k e_2=\big(\begin{smallmatrix}
0\\ \pm k
\end{smallmatrix}\big)$.

First, we still have 
$E(u_{ke_2}+u_{-ke_2})=\frac{1}{2}\sum_{x, y}\lambda_{x, y}\big[\big(\begin{smallmatrix}
0\\x
\end{smallmatrix}\big)+\big(\begin{smallmatrix}
0\\-x
\end{smallmatrix}\big)\big]\big(\begin{smallmatrix}
1&0\\-y&1
\end{smallmatrix}\big).$

Now, a calculation shows that 
\begin{align*}
\mbox{RHS of (\ref{eq: Z-case: E(E() )=E()E()})}=\sum_{x, y, b}\sum_{a>0}\lambda_{x, y}\lambda_{a, b}\big[\big( \begin{smallmatrix}
x+ay\\a
\end{smallmatrix}\big)+\big(\begin{smallmatrix}
-x-ay\\-a
\end{smallmatrix}\big) \big]\big(\begin{smallmatrix}
1-by&y\\
-b&1
\end{smallmatrix}\big).
\end{align*}
Note that we have used the fact $\lambda_{0, b}=0$ to cross out the terms corresponding to $a=0$.

Meanwhile, using the expression for $E(u_{ke_1}+u_{-ke_1})$, we can check the LHS of (\ref{eq: Z-case: E(E() )=E()E()}) equals 
\begin{align*}
   \frac{1}{2}\sum_{x, y}\lambda_{x, y}E\big[\big( \begin{smallmatrix}x+ky\\k \end{smallmatrix}\big) +\big(\begin{smallmatrix}-x-ky\\-k \end{smallmatrix}\big)  \big]\big(\begin{smallmatrix} 1&y\\0&1\end{smallmatrix}\big)
   +\frac{1}{2}\sum_{x, y}\lambda_{x, y}E\big[ \big(\begin{smallmatrix}-x+ky\\k \end{smallmatrix}\big)+\big(\begin{smallmatrix}x-ky\\-k \end{smallmatrix}\big) \big]\big(\begin{smallmatrix} 1&y\\0&1\end{smallmatrix}\big).
\end{align*}

To continue the calculation, we observe that the above two summands are equal by doing change of variables and applying the fact $\lambda_{x, y}=\lambda_{-x, y}$. Therefore,  
\begin{align*}
&\quad\mbox{LHS of (\ref{eq: Z-case: E(E() )=E()E()})}\\
&=\sum_{x, y}\lambda_{x, y}E\big[\big( \begin{smallmatrix}x+ky\\k \end{smallmatrix}\big) +\big(\begin{smallmatrix}-x-ky\\-k \end{smallmatrix}\big)  \big]\big(\begin{smallmatrix} 1&y\\0&1\end{smallmatrix}\big)\\
&=\sum_{k\mid x}\sum_y\lambda_{x, y}\big(\begin{smallmatrix}
\frac{x}{k}+y&\frac{x}{k}+y-1\\
1&1
\end{smallmatrix}\big)E(u_{ke_1}+u_{-ke_1})\big(\begin{smallmatrix}
\frac{x}{k}+y&\frac{x}{k}+y-1\\
1&1
\end{smallmatrix}\big)^{-1}\big(\begin{smallmatrix}
1&y\\0&1
\end{smallmatrix}\big)\\
&+\sum_{k\nmid x}\sum_y\big(\begin{smallmatrix}
\frac{x+ky}{d_x}&z_x\\
\frac{k}{d_x}&w_x
\end{smallmatrix}\big)E(u_{d_xe_1}+u_{-d_xe_1})\big(\begin{smallmatrix}
\frac{x+ky}{d_x}&z_x\\
\frac{k}{d_x}&w_x
\end{smallmatrix}\big)^{-1}\big(\begin{smallmatrix}
1&y\\0&1
\end{smallmatrix}\big)\\
&\quad(\mbox{in the 2nd summand, $d_x:=\gcd(k, |x|)$ and $(\frac{x+ky}{d_x})w_x-\frac{kz_x}{d_x}=1$ for some integers $w_x, z_x$.})\\
&=\sum_{k\mid x}\sum_{y,a,b}\lambda_{x, y}\lambda_{a, b}\big(\begin{smallmatrix}
\frac{x}{k}+y&\frac{x}{k}+y-1\\
1&1
\end{smallmatrix}\big)\big(\begin{smallmatrix}
a\\0
\end{smallmatrix}\big)\big(\begin{smallmatrix}
1&b\\0&1
\end{smallmatrix}\big)\big(
\begin{smallmatrix}
\frac{x}{k}+y&\frac{x}{k}+y-1\\
1&1
\end{smallmatrix}\big)^{-1}\big(\begin{smallmatrix}
1&y\\0&1
\end{smallmatrix}\big)\\
&\quad (\mbox{as $I_2=\emptyset$ and $I_1$ has no contribution to the sum.})\\
&=\sum_{k\mid x}\sum_{y, b}\sum_{a>0}\lambda_{x, y}\lambda_{a, b}\big[\big(\begin{smallmatrix}
a(\frac{x}{k}+y)\\
a
\end{smallmatrix}\big)+\big(\begin{smallmatrix}
-a(\frac{x}{k}+y)\\
-a
\end{smallmatrix}\big)\big]\big(\begin{smallmatrix}
1-b(\frac{x}{k}+y)& y+b(\frac{x^2}{k^2}+\frac{xy}{k})\\
-b&1+\frac{bx}{k}
\end{smallmatrix}\big).
\end{align*}
To get the last equality, we have used $\lambda_{0, b}=0$ for all $b\in \mathbb{Z}$ to cross out the terms corresponding to $a=0$.

Next, for each $(x, y, b)\in \Z^3$ and $0<a\in \Z$, by comparing the coefficients of (\ref{eq: Z-case: E(E() )=E()E()}) in front of the term 
$\big[\big(\begin{smallmatrix}
x+ay\\a
\end{smallmatrix}\big)+\big(\begin{smallmatrix}
-x-ay\\-a
\end{smallmatrix}\big)\big]\big(\begin{smallmatrix}
1-by&y\\-b&1
\end{smallmatrix}\big)$,  
we deduce that 
\begin{align}\label{eq: Z-case, 1st restriction}
\lambda_{x, y}\lambda_{a, b}=\begin{cases}
\lambda_{\frac{kx}{a}, y}\lambda_{a, b}, ~&\mbox{if}~bx=0~\mbox{and}~a\mid x\\
\lambda_{\frac{2k}{b}, y-\frac{4}{b}}\lambda_{a, -b}, ~&\mbox{if}~b\neq 0, b\mid 2,~\mbox{and}~2a+bx=0\\
0,~&\mbox{otherwise}.
\end{cases}
\end{align}

We are left to argue that $\lambda_{x, y}=0$ unless $(x,y)=(\pm k, 0)$.

First, observe that $\lambda_{x, y}\lambda_{a, b}=0$ if $x\neq 0$ and $|b|\geq 3$ by the last case in (\ref{eq: Z-case, 1st restriction}). We may take $y=b$, $a=x$ to deduce $\lambda_{x, b}=0$ for all $b$ and all $x\neq 0$ with $|b|\geq 3$ as $\lambda_{x, b}=\lambda_{-x, b}$. 

Next, we take $y=b\in \{\pm 2, \pm 1\}$ and $a=x>0$. Then (\ref{eq: Z-case, 1st restriction}) is simplified to 
\begin{align*}
\lambda_{x, 2}^2=0,~
\lambda_{x, -2}^2=\lambda_{-k, 0}\lambda_{x, 2}=\lambda_{-k, 0}0=0,
\lambda_{x, 1}^2=0,~
\lambda_{x, -1}^2=0.
\end{align*}
Therefore, $\lambda_{x, b}=0$ for all $x\neq 0$ and $|b|=1$ or $2$ .
To sum up, we have shown that $\lambda_{x, b}=0$ for all integers $b$ and $x$ with $bx\neq 0$. Recall that $\lambda_{0, b}=0$ for all $b\neq 0$ by assumption, we are left to determine $\lambda_{x,0}$. 

By plugging $b=0$ and $y=0$ into (\ref{eq: Z-case, 1st restriction}), we get 
$\lambda_{x, 0}\lambda_{a, 0}=\begin{cases}
\lambda_{\frac{kx}{a},0}\lambda_{a, 0},~&\mbox{if}~a\mid x\\
0, ~&\mbox{otherwise}.
\end{cases}$

Assume there exists some integer $n$ with $n\geq 2$ and $\lambda_{kn, 0}\neq 0$, then plug $a=kn$ into the above expression to deduce that 
$\lambda_{x, 0}=\begin{cases}
\lambda_{\frac{x}{n}, 0}, ~&\mbox{if}~(kn)\mid x\\
0,~&\mbox{otherwise}.
\end{cases}$

Next, plug $x=kn$ and $k$ respectively in the last expression, we get 
$0\neq \lambda_{kn, 0}=\lambda_{k, 0}=0$ as $n\geq 2$.
This gives a contradiction, and hence $\lambda_{kn,0}=0$ for all $|n|\geq 2$ as $\lambda_{x, y}=\lambda_{-x, y}$ for all $x$ and $y$.

Finally, we have shown that $\lambda_{x, y}=0$ unless $(x, y)=(\pm k, 0)$, which tells us that $E(u_{ke_1}+u_{-ke_1})=\lambda_{k,0}(u_{ke_1}+u_{-ke_1})$. This finishes the proof of Claim 1.
\end{proof}
Now, we consider the following index sets:
\begin{align*}
I_1'&:=\{1\leq i: E(u_{ie_1}+u_{-ie_1})=0\},\\
I_2'&:=\{1\leq i: E(u_{ie_1}+u_{-ie_1})=\lambda_i(u_{ie_1}+u_{-ie_1})~\mbox{for some}~\lambda_i\neq 0\}.
\end{align*}
Claim 1 implies that $\mathbb{N}^{+}=I_1'\sqcup I_2'$.
If $I_2'\neq \emptyset$, then we argue as in the proof of Claim 1 (while assuming $I_2\neq \emptyset$ there) to deduce that $P$ is of the required form. If $I_2'=\emptyset$, then $E(u_{ie_1}+u_{-ie_1})=0$ for all $i\geq 1$, which implies $E(u_{g.(ie_1)}+u_{g.(-ie_1)})=u_gE(u_{ie_1}+u_{-ie_1})u_g^*=0$ for all $g\in SL_2(\Z)$. As $\{g.(ie_1): g\in SL_2(\Z), i\geq 0\}=\Z^2$ and $\{u_v+u_{-v}: v\in\Z^2\}$ linearly spans a dense subset of $B$, we deduce that $P=L(PSL_2(\Z))=(B\cap L(0\Z^2))\rtimes PSL_2(\Z)$. 
\end{proof}
Now, let us prove Theorem \ref{thm: SL_2(Z)-thm} using Theorem \ref{thm: PSL_2(Z)-thm}.
\begin{proof}[Proof of Theorem \ref{thm: SL_2(Z)-thm}]
Let $q=\frac{u_{id}+u_{-id}}{2}$ as defined in the proof of Corollary \ref{cor: max Haagerup for LG}. Since $q$ lies in the center of $B\rtimes SL_2(\Z)$ and $q(B\rtimes SL_2(\Z))\cong B\rtimes PSL_2(\Z)$, we know that 
$\pi: (qL(SL_2(\Z))\subset qP\subset q(B\rtimes SL_2(\Z)))\cong (L(PSL_2(\Z))\subset \pi(qP)\subset B\rtimes PSL_2(\Z))$, where $\pi$ is defined in the proof of Corollary \ref{cor: max Haagerup for LG}. Then by Theorem \ref{thm: PSL_2(Z)-thm}, we know that $\pi(qP)=(B\rtimes L(n\Z^2))\rtimes PSL_2(\Z)$ for some integer $n\geq 0$. From the definition of $\pi$, we deduce that $qP=q[(B\cap L(n\Z^2))\rtimes SL_2(\Z)]$.

Let $E$ denotes the trace preserving conditional expectation from $(L(\Z^2\rtimes SL_2(\Z)),\tau)$ onto $(P, \tau|_P)$. The key step is to prove the following claim:

\textbf{Claim 1:} For each $k\geq 1$, $E(u_{ke_1})=\lambda_k(u_{ke_1}+u_{-ke_1})+\mu_k(u_{ke_1}+u_{-ke_1})u_{-id}$ for some scalars $\lambda_k$ and $\mu_k$.
\begin{proof}[Proof of Claim 1]
We follow the notations and the proof of Theorem \ref{thm: complete description of intermediate vn algs} closely and do necessary modification.

\textbf{Case $k=1$.}
Let us comment on how to modify the steps in the proof of Theorem \ref{thm: complete description of intermediate vn algs}.

\textbf{Step 1}: we do not make any change. In particular, (\ref{eq: relation btw c_a and c_-a})(\ref{eq: relation btw c_g.a and c_a})(\ref{eq: key restriction for E(c_a) in proving maximality}) still hold but for all $v, w\in \Z^2$. In particular, we can still use (\ref{eq: def of c_e_1}) for the expression of $c_{e_1}$ while keeping in mind that $\lambda_{x, y}=0=\mu_{x, y}$ if $x$ or $y\in \Q\setminus \Z$.

\textbf{Step 2}: we do not make any change. 

\textbf{Step 3}: (\ref{eq: lambda is symmetric along x=-1}) still holds. After that, we argue as follows:

First, since $qE(u_{e_1})\in qP=q[(B\cap L(n\Z^2))\rtimes SL_2(\Z)]$ and $\{u_{n\ell e_1}+u_{-n\ell e_1}: \ell\in\Z\}$ linearly spans a dense subset of $B\cap L(n\Z^2)$, it is not hard to check that these imply that 
\begin{align*}
\lambda_{x-1, y}+\mu_{x-1,-y}=0,~\mbox{for all $x\in \Z\setminus n\Z$ and all $y\in\Z$.}
\end{align*}

Second, notice that $\langle u_{e_1}-E(u_{e_1}), qP\rangle=0$ and $(u_{nme_1}+u_{-nme_1})u_gq\in qP$ for all $m\in\Z$, one has
$\langle u_{e_1}-E(u_{e_1}), (u_{nme_1}+u_{-nme_1})u_gq\rangle=0.$ Then, for each $0\neq s\in\Z$, we plug $g=\big(\begin{smallmatrix}
1&s\\0&1
\end{smallmatrix}\big)$ into the above expression and use (\ref{eq: lambda is symmetric along x=-1}) to get 
\begin{align*}
\lambda_{x-1, s}+\mu_{x-1, -s}=0~\mbox{for all $x\in n\Z$ and all $0\neq s\in\Z$}. 
\end{align*}
Similarly, if we set $g=id$, then we can deduce that  
\begin{align*}
\lambda_{x-1, 0}+\mu_{x-1, 0}=0,~\mbox{if $\pm 1\neq x\in n\Z$},\\
\lambda_{x-1, 0}+\mu_{x-1, 0}=\frac{1}{2},~\mbox{if $\pm 1=x\in n\Z$}.
\end{align*}

To sum up, we have proved that
\begin{align}\label{eq: Z case and k=1, relation between lambda and mu by symmetry}
\begin{split}
\lambda_{x, y}+\mu_{x, -y}&=0, ~\forall x,\forall y\neq 0,\\
\lambda_{x, 0}+\mu_{x, 0}&=0,~\forall (x+1)\in\Z\setminus n\Z, \\
\lambda_{x, 0}+\mu_{x, 0}&=0,~\forall (x+1)\in n\Z\setminus\{\pm 1\}.
\end{split}
\end{align}

\textbf{Step 4}: notice that (\ref{eq: lambda_a, b=0 for a neq -1 and b nonzero}) there still works since $\lambda_{x, y}+\mu_{x, -y}=0$ still holds true for all $y\neq 0$ and this is the only part of (\ref{eq: relation between lambda and mu by symmetry}) needed for the proof of (\ref{eq: lambda_a, b=0 for a neq -1 and b nonzero}).

Next, observe that for the proof of the rest part of Step 4 there, the weaker version of (\ref{eq: relation between lambda and mu by symmetry}), i.e. $\lambda_{x, 0}+\mu_{x, 0}=0$ for all $x\not\in\{-2,-1,0\}$ is needed. Note that this weaker version still holds by (\ref{eq: Z case and k=1, relation between lambda and mu by symmetry}) for all $n\geq 0$. 

Therefore, Claim 1 holds for $k=1$.


\textbf{Induction step on $k$:}
assume now that Claim 1 holds true for $k=1,\cdots,\ell-1$, let us check it holds for $k=\ell$.
The proof is essentially the same as above, but notations are much more involved. We decide to postpone it to the appendix in the end of the paper.
\end{proof}
From Claim 1, we deduce $(1-q)E(u_{ke_1})=(\lambda_k-\mu_k)(u_{ke_1}+u_{-ke_1})(1-q)$ for each $k\geq 1$. 
We are very grateful to the referee for suggesting the following argument, which simplified our original argument.

Write $c_k:=\lambda_k-\mu_k$ for each $k\geq 1$. By adding the conjugation of the above identity by $g=-id$, we get $(1-q)E(u_{ke_1}+u_{-ke_1})=2(1-q)c_k(u_{ke_1}+u_{-ke_1})$. Since $(1-q)E$ is a projection, we deduce $2c_k=0$ or $1$. By taking a conjugation by any $g\in SL_2(\Z)$, we deduce that $(1-q)E(u_v+u_{-v})=(1-q)(u_v+u_{-v})$ or $0$ for all $v\in \Z^2$. As $L(SL_2(\Z))\subset P$, this concludes that $(1-q)P=(1-q)[(B\cap P)\rtimes SL_2(\Z)]$ and $B\cap P=B\cap L(m\Z^2)$, where $m$ is the smallest $k\geq 1$ with $c_k\neq 0$ (if such $k$ does not exists, take $m=0$). 
\end{proof}

\appendix\label{appendix}
\section{{Induction step in the proof of Theorem \ref{thm: SL_2(Z)-thm}}}

In this appendix, we prove the induction step needed for the proof of Claim 1, Theorem \ref{thm: SL_2(Z)-thm}.

Recall that we assume $E(u_{ke_1})=\lambda_k(u_{ke_1}+u_{-ke_1})+\mu_k(u_{ke_1}+u_{-ke_1})u_{-id}$ holds for all $k=1,\cdots, \ell-1$. From the definition of $c_v$, i.e. $c_v:=u_v^*E(u_v)$, we know that $c_{ke_1}=\lambda_k(u_{0e_1}+u_{-2ke_1})+\mu_k(u_{0e_1}+u_{-2ke_1})u_{-id}$ for all $1\leq k<\ell$.
We aim to show this identity also holds for some scalars $\lambda_k$, $\mu_k$ when $k=\ell$. 

We prepare the corresponding steps as in the proof of Theorem \ref{thm: complete description of intermediate vn algs}.

\textbf{Step 1:} it is not hard to see we can still write 
\begin{align}\label{def: Z-case, c_ell e_1}
c_{\ell e_1}=\sum_{x, y}\lambda_{x, y}\big(\begin{smallmatrix}
x\\
0
\end{smallmatrix}\big)\big(\begin{smallmatrix}
1&y\\
0&1
\end{smallmatrix}\big)+\sum_{x, y}\mu_{x, y}\big(\begin{smallmatrix}
x\\
0
\end{smallmatrix}\big)\big(\begin{smallmatrix}
-1&y\\
0&-1
\end{smallmatrix}\big).
\end{align}
We remind the reader that in the above expression, one can think of $(x, y)$ belongs to $\mathbb{Q}^2$, but both $\lambda_{x, y}$ and $\mu_{x, y}$ are zero unless $(x, y)\in\mathbb{Z}^2$.
More importantly, the identities (\ref{eq: relation btw c_a and c_-a}), (\ref{eq: relation btw c_g.a and c_a}) and (\ref{eq: key restriction for E(c_a) in proving maximality}) still hold.
Our goal is to show that $\lambda_{x, y}=0=\mu_{x, y}$ unless $(x, y)=(0, 0)$, $(-2\ell, 0)$.

\textbf{Step 2:} from $\langle u_{\ell e_1}-E(u_{\ell e_1}), P\rangle=0$ and $u_g\in P$, we infer that
\begin{align}\label{eq: Z case and k=ell, lambda_-1, b=0=mu_-1, b}
\lambda_{-\ell, b}=0=\mu_{-\ell, b}~\mbox{for all}~b.
\end{align}

Using $\sigma_{\ell e_1}(c_{\ell e_1})=c_{-\ell e_1}^*$, we can deduce the following analogue of (\ref{eq: conjugation restriction on coefficients}): 
\begin{align*}
\lambda_{x, y}=\overline{\lambda_{x,-y}}~\mbox{and}~\mu_{x, y}=\overline{\mu_{-2\ell-x,-y}}~\mbox{for all}~x, y. 
\end{align*}

Next, let us compute both sides of (\ref{eq: key restriction for E(c_a) in proving maximality}) by setting $g=id$, $v=\ell e_1$ and $w=\ell e_2$.
\begin{align*}
\mbox{RHS of (\ref{eq: key restriction for E(c_a) in proving maximality})}
&=\sum_{x, y, a, b}\lambda_{x, y}\lambda_{a, b}\big(\begin{smallmatrix}
x+\ell+(a+\ell)y\\ \ell+a
\end{smallmatrix}\big)\big(\begin{smallmatrix}
1-by&y\\-b&1
\end{smallmatrix}\big)+\sum_{x, y, a, b}\lambda_{x, y}\mu_{a, b}\big(\begin{smallmatrix}
x+\ell+(a+\ell)y\\ \ell+a
\end{smallmatrix}\big)\big(\begin{smallmatrix}
-1-by&-y\\-b&-1
\end{smallmatrix}\big)\\
&+\sum_{x, y, a, b}\mu_{x, y}\lambda_{a, b}\big(\begin{smallmatrix}
x+\ell+(a+\ell)y\\ -\ell-a
\end{smallmatrix}\big)\big(\begin{smallmatrix}
-1-by&y\\b&-1
\end{smallmatrix}\big)+\sum_{x, y, a, b}\mu_{x, y}\mu_{a, b}\big(\begin{smallmatrix}
x+\ell+(a+\ell)y\\ -\ell-a
\end{smallmatrix}\big)\big(\begin{smallmatrix}
1-by&-y\\b&1
\end{smallmatrix}\big).\\
\mbox{LHS of (\ref{eq: key restriction for E(c_a) in proving maximality})}
&=\sum_{x, y}\lambda_{x, y}E\big[\big(\begin{smallmatrix}
\ell+x+\ell y\\ \ell
\end{smallmatrix}\big)\big]\big(\begin{smallmatrix}
1&y\\0&1
\end{smallmatrix}\big)+\sum_{x, y}\mu_{x, y}E\big[\big(\begin{smallmatrix}
\ell+x+\ell y\\-\ell
\end{smallmatrix}\big)\big]\big(\begin{smallmatrix}
-1&y\\0&-1
\end{smallmatrix}\big).
\end{align*}
Now, we need to compute each summand of the above expression. 

For the 1st summand, we may use the definition $E(u_v)=u_vc_v$ and (\ref{eq: relation btw c_g.a and c_a}) to deduce that
\begin{align*}
 \sum_{x, y}\lambda_{x, y}E\big[\big(\begin{smallmatrix}
\ell+x+\ell y\\ \ell
\end{smallmatrix}\big)\big]\big(\begin{smallmatrix}
1&y\\0&1
\end{smallmatrix}\big)
&=\sum_{\ell\mid x}\sum_y\lambda_{x, y}\big(\begin{smallmatrix}
\ell+x+\ell y\\ \ell
\end{smallmatrix}\big)\big(\begin{smallmatrix}
\frac{x}{\ell}+1+y&\frac{x}{\ell}+y\\
1&1
\end{smallmatrix}\big)c_{\ell e_1}\big(\begin{smallmatrix}
\frac{x}{\ell}+1+y&\frac{x}{\ell}+y\\
1&1
\end{smallmatrix}\big)^{-1}\big(\begin{smallmatrix}
1&y\\0&1
\end{smallmatrix}\big)\\
&+\sum_{\ell\nmid x}\sum_y\lambda_{x, y}\big(\begin{smallmatrix}
x+\ell+\ell y\\ \ell
\end{smallmatrix}\big)\big(\begin{smallmatrix}
\frac{x+\ell+\ell y}{d_x}&z_x\\
\frac{\ell}{d_x}&w_x
\end{smallmatrix}\big)c_{d_xe_1}\big(\begin{smallmatrix}
\frac{x+\ell+\ell y}{d_x}&z_x\\
\frac{\ell}{d_x}&w_x
\end{smallmatrix}\big)^{-1}\big(\begin{smallmatrix}
1&y\\
0&1
\end{smallmatrix}\big).
\end{align*}
Here, $d_x:=\gcd(\ell, |x|)$ and $(\frac{x+\ell+\ell y}{d_x})w_x-\frac{z_x\ell}{d_x}=1$ for some integers $w_x$, $z_x$.

Next, by induction hypothesis, $c_{de_1}=\lambda_d(u_{0e_1}+u_{-2de_1})+\mu_d(u_{0e_1}+u_{-2de_1})u_{-id},~\forall~d<\ell$. Hence, we can substitute this expression and (\ref{def: Z-case, c_ell e_1}) for $c_{de_1} (\forall~d<\ell)$ and $c_{\ell e_1}$ respectively, and then split $\sum_{\ell\nmid x}$ as $\sum_{1\leq d<\ell}\sum_{x, gcd(\ell, |x|)=d}$. Finally, we arrive at
\begin{align*}
\sum_{x, y}\lambda_{x, y}E\big[\big(\begin{smallmatrix}
\ell+x+\ell y\\ \ell
\end{smallmatrix}\big)\big]\big(\begin{smallmatrix}
1&y\\0&1
\end{smallmatrix}\big)
&=\sum_{\ell\mid x}\sum_{y,a,b}\lambda_{x, y}\lambda_{a, b}\big(\begin{smallmatrix}
(x+\ell+\ell y)(1+\frac{a}{\ell})\\ \ell+a
\end{smallmatrix}\big)\big(\begin{smallmatrix}
1-b(\frac{x}{\ell}+1+y)&y+b(\frac{x}{\ell}+1)(\frac{x}{\ell}+1+y)\\
-b&b(\frac{x}{\ell}+1)+1
\end{smallmatrix}\big)\\
&+\sum_{\ell\mid x}\sum_{y,a,b}\lambda_{x, y}\mu_{a,b}\big(\begin{smallmatrix}
(x+\ell+\ell y)(1+\frac{a}{\ell})\\ \ell+a
\end{smallmatrix}\big)\big(\begin{smallmatrix}
-1-b(\frac{x}{\ell}+1+y)&-y+b(\frac{x}{\ell}+1)(\frac{x}{\ell}+1+y)\\
-b&b(\frac{x}{\ell}+1)-1
\end{smallmatrix}\big)\\
&+\sum_{1\leq d<\ell}\sum_{\substack{x, y\\ \gcd(|x|, \ell)=d}}\lambda_{x, y}\lambda_d\big[\big(\begin{smallmatrix}
x+\ell+\ell y\\ \ell
\end{smallmatrix}\big)+\big(\begin{smallmatrix}
-(x+\ell+\ell y)\\ -\ell
\end{smallmatrix}\big)\big]\big(\begin{smallmatrix}
1&y\\0&1
\end{smallmatrix}\big)\\
&+\sum_{1\leq d<\ell}\sum_{\substack{x, y\\ \gcd(|x|, \ell)=d}}\lambda_{x, y}\mu_d\big[\big(\begin{smallmatrix}
x+\ell+\ell y\\ \ell
\end{smallmatrix}\big)+\big(\begin{smallmatrix}
-(x+\ell+\ell y)\\ -\ell
\end{smallmatrix}\big)\big]\big(\begin{smallmatrix}
-1&-y\\0&-1
\end{smallmatrix}\big).
\end{align*}
Similarly, for the 2nd summand in the LHS of (\ref{eq: key restriction for E(c_a) in proving maximality}), we have
\begin{align*}
\sum_{x, y}\mu_{x, y}E\big[\big(\begin{smallmatrix}
\ell+x+\ell y\\-\ell
\end{smallmatrix}\big)\big]\big(\begin{smallmatrix}
-1&y\\0&-1
\end{smallmatrix}\big)
&=\sum_{\ell\mid x}\sum_{y,a,b}\mu_{x, y}\lambda_{a, b}\big(\begin{smallmatrix}
(x+\ell+\ell y)(1+\frac{a}{\ell})\\ -(\ell+a)
\end{smallmatrix}\big)\big(\begin{smallmatrix}
-1-b(\frac{x}{\ell}+1+y)&y-b(\frac{x}{\ell}+1)(\frac{x}{\ell}+1+y)\\
b&b(\frac{x}{\ell}+1)-1
\end{smallmatrix}\big)\\
&+\sum_{\ell\mid x}\sum_{y,a,b}\mu_{x, y}\mu_{a,b}\big(\begin{smallmatrix}
(x+\ell+\ell y)(1+\frac{a}{\ell})\\ -(\ell+a)
\end{smallmatrix}\big)\big(\begin{smallmatrix}
1-b(\frac{x}{\ell}+1+y)&-y-b(\frac{x}{\ell}+1)(\frac{x}{\ell}+1+y)\\
b&b(\frac{x}{\ell}+1)+1
\end{smallmatrix}\big)\\
&+\sum_{1\leq d<\ell}\sum_{\substack{x, y\\ \gcd(|x|, \ell)=d}}\mu_{x, y}\lambda_d\big[\big(\begin{smallmatrix}
x+\ell+\ell y\\ -\ell
\end{smallmatrix}\big)+\big(\begin{smallmatrix}
-(x+\ell+\ell y)\\ \ell
\end{smallmatrix}\big)\big]\big(\begin{smallmatrix}
-1&y\\0&-1
\end{smallmatrix}\big)\\
&+\sum_{1\leq d<\ell}\sum_{\substack{x, y\\ \gcd(|x|, \ell)=d}}\mu_{x, y}\mu_d\big[\big(\begin{smallmatrix}
x+\ell+\ell y\\ -\ell
\end{smallmatrix}\big)+\big(\begin{smallmatrix}
-(x+\ell+\ell y)\\ \ell
\end{smallmatrix}\big)\big]\big(\begin{smallmatrix}
1&-y\\0&1
\end{smallmatrix}\big).
\end{align*}

For any given $s$, $a\neq -\ell$, $b$ and $y$ in $\Z$, by computing the coefficients of the term $\big(\begin{smallmatrix}
s\\a+\ell
\end{smallmatrix}\big)\big(\begin{smallmatrix}
1-by&y\\-b&1
\end{smallmatrix}\big)$ on both sides of (\ref{eq: key restriction for E(c_a) in proving maximality}),
we get the following identity:
\begin{align}\label{system of equtions: Z-case and k=ell, eq3-eq4}
\begin{split}
&\quad\quad\lambda_{s-\ell-(a+\ell)y, y}\lambda_{a, b}+\mu_{s-\ell-(a+\ell)y, -y}\mu_{-2\ell-a, -b}\\
&=\begin{cases}
\lambda_{\ell(\frac{s}{a+\ell}-1-y), y}\lambda_{a, 0}+\mu_{\ell(y-1-\frac{s}{a+\ell}), -y}\mu_{-a-2\ell, 0},~&\mbox{if}~b=0, s\neq 0, a\neq 0, -\ell, -2\ell~\mbox{and}~(a+\ell)\mid s\\
0,~&\mbox{if}~b=0, s\neq 0, a\neq 0, -\ell, -2\ell~\mbox{and}~(a+\ell)\nmid s\\
0, ~&\mbox{if}~bs\neq 0~\mbox{and}~y\neq \frac{s}{a+\ell}+\frac{2}{b}, \frac{s}{a+\ell}.
\end{cases}
\end{split}
\end{align}

Similarly, for any given $s$, $a\neq -\ell$, $b$ and $y$ in $\Z$, we compare the coefficients of the term $\big(\begin{smallmatrix}
s\\-a-\ell
\end{smallmatrix}\big)\big(\begin{smallmatrix}
-1-by&y\\ b&-1
\end{smallmatrix}\big)$ on both sides of (\ref{eq: key restriction for E(c_a) in proving maximality}). The following identity holds.
\begin{align}\label{system of equtions: Z-case and k=ell, eq7-eq8}
\begin{split}
&\quad\quad \lambda_{s-\ell-(a+\ell)y, -y}\mu_{-2\ell-a, -b}+\mu_{s-\ell-(a+\ell)y, y}\lambda_{a, b}\\
&=\begin{cases}
\lambda_{\ell(y-1-\frac{s}{a+\ell}), -y}\mu_{-a-2\ell, 0}+\mu_{\ell(\frac{s}{a+\ell}-1-y), y}\lambda_{a, 0},~&\mbox{if}~b=0, s\neq 0, a\neq 0, -\ell, -2\ell~\mbox{and}~(a+\ell)\mid s\\
0,~&\mbox{if}~b=0, s\neq 0, a\neq 0, -\ell, -2\ell~\mbox{and}~(a+\ell)\nmid s\\
0, ~&\mbox{if}~bs\neq 0~\mbox{and}~y\neq \frac{s}{a+\ell}-\frac{2}{b}, \frac{s}{a+\ell}.
\end{cases}
\end{split}
\end{align}

\textbf{Step 3:} similar to (\ref{eq: lambda is symmetric along x=-1}), we have
\begin{align}\label{eq: Z case and k=ell, lambda is symmetric along x=-1}
\lambda_{x, y}=\lambda_{-2\ell-x, y}~\mbox{and}~\mu_{x, y}=\mu_{-2\ell-x, y},~\forall x,~\forall y.
\end{align}

Next, from $qE(u_{\ell e_1})\in qP=q[(B\cap L(n\Z^2))\rtimes SL_2(\Z)]$ and (\ref{eq: Z case and k=ell, lambda is symmetric along x=-1}), we can deduce that 
\[\lambda_{x, y}+\mu_{x, -y}=0,~\forall~(x+\ell)\in\Z\setminus n\Z,~\forall~y. \]
Similarly, from $\langle u_{\ell e_1}-E(u_{\ell e_1}), qP\rangle=0$, $(u_{nte_1}+u_{-nte_1})u_gq\in qP$ for all $t\in\Z$ and (\ref{eq: Z case and k=ell, lambda is symmetric along x=-1}), we deduce that
\begin{align*}
\lambda_{x, y}+\mu_{x, -y}&=0, ~\forall~(x+\ell)\in n\Z, \forall~y\neq 0.\\
\lambda_{x, 0}+\mu_{x, 0}&=0,~\forall~(x+\ell)\in n\Z\setminus\{\pm \ell\}.\\
\lambda_{x, 0}+\mu_{x, 0}&=\frac{1}{2},~\mbox{if}~x=0, -2\ell~\mbox{and}~ \ell\in n\Z.
\end{align*}
To sum up, we have shown that
\begin{align}\label{eq: Z case and k=ell, relation between lambda and mu by symmetry}
\begin{split}
\lambda_{x, y}+\mu_{x, -y}&=0, ~\forall~x, \forall~y\neq 0.\\
\lambda_{x, 0}+\mu_{x, 0}&=0,~\forall~x\in \Z\setminus \{0, -2\ell\}.
\end{split}
\end{align}

\textbf{Step 4:} show that $\lambda_{x, y}=0=\mu_{x, y}$ unless $(x, y)=(0, 0)$, $(-2\ell, 0)$.

By (\ref{eq: Z case and k=ell, lambda is symmetric along x=-1}) and (\ref{eq: Z case and k=ell, relation between lambda and mu by symmetry}), we can simplify the last case in both (\ref{system of equtions: Z-case and k=ell, eq3-eq4}) and (\ref{system of equtions: Z-case and k=ell, eq7-eq8}) to the following:
\begin{align*}
\lambda_{s-\ell-(a+\ell)y, y}\lambda_{a, b}&=0,~\forall~bs\neq 0,~\forall~a\neq -\ell,~\forall~y\neq 0,~ \frac{s}{a+\ell},~\frac{s}{a+\ell}+\frac{2}{b},\\
\lambda_{s-\ell-(a+\ell)y, -y}\lambda_{a, b}&=0,~\forall~bs\neq 0,~\forall~a\neq -\ell,~\forall~y\neq 0,~ \frac{s}{a+\ell},~\frac{s}{a+\ell}-\frac{2}{b}.
\end{align*}

By a similar argument as we used to prove (\ref{eq: lambda_a, b=0 for a neq -1 and b nonzero}), we can check that $\lambda_{a, b}=0$ for all $a\neq -\ell$ and $b\neq 0$. 
By combining this with (\ref{eq: Z case and k=ell, lambda_-1, b=0=mu_-1, b}), we have shown that
\[\lambda_{a, b}=0~\mbox{unless $b=0$.} \]

Now, we solve for $\lambda_{a, 0}$.

We plug $y=0$ in the 1st and 2nd case in (\ref{system of equtions: Z-case and k=ell, eq3-eq4})-(\ref{system of equtions: Z-case and k=ell, eq7-eq8}), and use  (\ref{eq: Z case and k=ell, lambda is symmetric along x=-1})-(\ref{eq: Z case and k=ell, relation between lambda and mu by symmetry}) to get a single expression:
\begin{align*}
&\quad \lambda_{s-\ell, 0}\lambda_{a, 0}-\mu_{s-\ell, 0}\lambda_{a, 0}\\
&=\begin{cases}
\lambda_{\ell(\frac{s}{a+\ell}-1), 0}\lambda_{a, 0}-\mu_{\ell(\frac{s}{a+\ell}-1), 0}\lambda_{a, 0},~&\mbox{if}~s\neq0, a\neq 0, -\ell,-2\ell~\mbox{and}~(a+\ell)\mid s\\
0,~&\mbox{if}~s\neq0, a\neq 0, -\ell,-2\ell~\mbox{and}~(a+\ell)\nmid s.
\end{cases}
\end{align*}
Assume $\lambda_{a, 0}\neq 0$ for some $a\neq 0, -\ell, -2\ell$, then 
\begin{align*}
 \lambda_{s-\ell, 0}-\mu_{s-\ell, 0}
=\begin{cases}
\lambda_{\ell(\frac{s}{a+\ell}-1), 0}-\mu_{\ell(\frac{s}{a+\ell}-1), 0},~&\mbox{if}~s\neq0~\mbox{and}~(a+\ell)\mid s\\
0,~&\mbox{if}~s\neq0~\mbox{and}~(a+\ell)\nmid s.
\end{cases}
\end{align*}
By change of variables, this means that
\begin{align}\label{eq: Z case and final eq to solve for lambda_a, 0}
\begin{split}
\lambda_{s-\ell, 0}&=\mu_{s-\ell, 0}~\mbox{for all $s\neq 0$ with $(a+\ell)\nmid s$},\\
\lambda_{t(a+\ell)-\ell, 0}-\mu_{t(a+\ell)-\ell, 0}&=\lambda_{\ell(t-1), 0}-\mu_{\ell(t-1), 0}~\mbox{for all $0\neq t\in\Z$.}
\end{split}
\end{align}
Next, we split the proof by considering two cases.

\textbf{Case 1}: $(a+\ell)\nmid \ell$.

Plug $s=\ell$ in (\ref{eq: Z case and final eq to solve for lambda_a, 0}), we get $\lambda_{0, 0}=\mu_{0, 0}$. Plug $t=1$ in (\ref{eq: Z case and final eq to solve for lambda_a, 0}), we get $\lambda_{a, 0}-\mu_{a, 0}=\lambda_{0, 0}-\mu_{0, 0}=0$. So $\lambda_{a, 0}=\mu_{a, 0}=0$ by (\ref{eq: Z case and k=ell, relation between lambda and mu by symmetry}), a contradiction.

\textbf{Case 2}: $(a+\ell)\mid \ell$.

Notice that by (\ref{eq: Z case and k=ell, lambda is symmetric along x=-1}), we may further assume $a+\ell>0$.
Write $\ell=(a+\ell)d$ for some integer $d$. Note that $1<d\leq \ell$ as $a\neq 0$.
Then, by using $\ell=(a+\ell)d$, the 2nd identity in (\ref{eq: Z case and final eq to solve for lambda_a, 0}) can be written as $\lambda_{(t-d)(a+\ell), 0}-\mu_{(t-d)(a+\ell), 0}=\lambda_{(td-d)(a+\ell), 0}-\mu_{(td-d)(a+\ell), 0}$ for all $0\neq t\in \Z$. Thus after replacing $t$ by $td$, $td^2$,$\cdots$, we can deduce that
\begin{align*}
&\quad\lambda_{(t-d)(a+\ell), 0}-\mu_{(t-d)(a+\ell), 0}\\
&=\lambda_{(td-d)(a+\ell), 0}-\mu_{(td-d)(a+\ell), 0}\\
&=\lambda_{(td^2-d)(a+\ell), 0}-\mu_{(td^2-d)(a+\ell), 0}\\
&=\cdots\\
&=\lambda_{(td^k-d)(a+\ell), 0}-\mu_{(td^k-d)(a+\ell), 0},~\forall~k\geq 1,~\forall~0\neq t\in\Z.
\end{align*}
Notice that $(td^k-d)(a+\ell)\neq (td^{k'}-d)(a+\ell)$ for all $k\neq k'$ as $t\neq 0$ and $d>1$.
Argue as in the proof of Step 4 in Theorem \ref{thm: complete description of intermediate vn algs}, this implies that $\lambda_{(t-d)(a+\ell), 0}=\mu_{(t-d)(a+\ell), 0}$ for all $0\neq t\in\Z$. In particular, take $t=1$, we get $\lambda_{a, 0}=\mu_{a, 0}$, hence $\lambda_{a, 0}=0$ by (\ref{eq: Z case and k=ell, relation between lambda and mu by symmetry}), a contradiction.

Therefore, we have shown that $\lambda_{a, 0}=0$ for all $a\neq 0$, $-\ell$, $-2\ell$. By combining this with (\ref{eq: Z case and k=ell, lambda_-1, b=0=mu_-1, b}) and (\ref{eq: Z case and k=ell, relation between lambda and mu by symmetry}), we finish the proof.
\subsection*{Acknowledgements} Part of this work was done when Y.J. was a postdoc in IMPAN, Poland, where he was partially supported by the National Science Center (NCN) grant no. 2014/14/E/ST1/00525. He thanks Prof. Adam Skalski for helpful discussion and support there. He is also supported by ``the Fundamental Research Funds for the Central Universities" (Grant No.DUT19RC(3)075). We also thank the referee for his/her critical comments and helpful suggestions which help improve the presentation of the paper greatly.

\end{document}